\documentclass[a4paper,10pt]{amsart}
\usepackage{amsmath,amsfonts,amssymb,amsthm, amsfonts, amscd, url}
\usepackage[utf8]{inputenc}
\usepackage[english]{babel}
\usepackage{url}

\usepackage[cmtip,arrow]{xy}

\usepackage{pb-diagram,pb-xy}

\swapnumbers 
\theoremstyle{plain}
\newtheorem{thm}{Theorem}[section]
\newtheorem{prop}[thm]{Proposition}
\newtheorem{lemma}[thm]{Lemma}
\newtheorem{cor}[thm]{Corollary}
\theoremstyle{remark}
\theoremstyle{definition}
\newtheorem{rem}[thm]{Remark}
\newtheorem{rems}[thm]{Remarks}
\newtheorem{remdef}[thm]{Remark-Definition}
\newtheorem{remsdefs}[thm]{Remarks-Definitions}

\newtheorem{defi}[thm]{Definition}

\newtheorem{notas}[thm]{Notations}
\newtheorem{nota}[thm]{Notation}
\newtheorem{examples}[thm]{Examples}

\usepackage[english]{babel}

\title[Generalized principal logarithms]{SVD-closed subgroups of the unitary group:\\ generalized principal logarithms\\ and minimizing geodesics}

\author{Donato Pertici \and Alberto Dolcetti}

\begin{document}

\parindent 0pt
\selectlanguage{english}

\maketitle

\vspace*{-0.2in}

\begin{center}
{\scriptsize Dipartimento di Matematica  e Informatica, Viale Morgagni 67/a, 50134 Firenze, ITALIA

\vspace*{0.07in}

donato.pertici@unifi.it,  \   http://orcid.org/0000-0003-4667-9568

\vspace*{-0.03in}

alberto.dolcetti@unifi.it, \  http://orcid.org/0000-0001-9791-8122
}

\end{center}


\vspace*{-0.1in}

\begin{abstract}
We study the set of generalized principal $\mathfrak{g}$-logarithms of any matrix belonging to a connected SVD-closed subgroup $G$ of $U_n$, with Lie algebra $\mathfrak{g}$. This set  is a non-empty disjoint union of a finite number of subsets diffeomorphic to homogeneous spaces, and it is related to a suitable set of minimizing geodesics. Many particular cases for the group $G$ are explicitly analysed.
\end{abstract}


{\small \tableofcontents}

\renewcommand{\thefootnote}{\fnsymbol{footnote}}

\renewcommand{\thefootnote}{\arabic{footnote}}
\setcounter{footnote}{0}

\vspace*{-0.3in}

{\small {\scshape{Keywords.}} Generalized principal logarithm; SVD-decomposition, SVD-closed subgroup, Frobenius metric, minimizing geodesics, (symmetric) homogeneous space.
\smallskip

{\small {\scshape{Mathematics~Subject~Classification~(2020):}}  53C30, 15B30, 22E15.

{\small {\scshape{Grants:}}
This research has been partially supported by GNSAGA-INdAM (Italy).

\section*{Introduction}\label{intro}

If $M$ is a matrix belonging to a connected closed subgroup $G$ of $GL_n (\mathbb{C})$, having $\mathfrak{g}$ as Lie algebra, we say that a matrix $L \in \mathfrak{g}$ is a \emph{generalized principal} $\mathfrak{g}$-\emph{logarithm} of $M$, if $\exp(L) = M$ and  $- \pi \le Im(\lambda) \le\pi$, for every eigenvalue $\lambda$ of $L$; the set of all  generalized  principal 
$\mathfrak{g}$-logarithms of $M$ is denoted by $\mathfrak{g}$--$plog(M)$. Our definition  relaxes the usual one of \emph{principal logarithm}, which excludes the matrices $M \in GL_n(\mathbb{C})$ with negative eigenvalues (see, for instance, \cite[Thm.\,1.31]{Hi2008}). 
The usual definition implies both existence and uniqueness of a principal logarithm.
In some relevant cases, matrices with negative eigenvalues and belonging to a closed subgroup $G$ of $GL_n (\mathbb{C})$, have an infinite set of generalized  principal $\mathfrak{g}$--logarithms, on which it is possible to define some natural geometric structures. We have already studied the sets $\mathfrak{so}_n$--$plog(M)$, if $M \in SO_n$, and $\mathfrak\mathfrak{gl}_n (\mathbb{R})$--$plog(M)$, if $M$ is semi-simple (see \cite{DoPe2018a} and \cite{Pe2022}).
Our interest in the set $\mathfrak{g}$--$plog(M)$ is related to a differential-geometric setting, which we briefly describe.
Denote by ${\phi}$  the \emph{Frobenius} (or \emph{Hilbert-Schmidt}) positive definite real scalar product on $\mathfrak{gl}_n (\mathbb{C})$, defined by ${\phi}(A, B) := Re(tr(A B ^*))$. If $G$ is a connected closed subgroup $G$ of the unitary group $U_n$ (with Lie algebra $\mathfrak{g}$), we still denote by $\phi$ the Riemannian metric on $G$, obtained by restriction of the Frobenius scalar product of $\mathfrak{gl}_n (\mathbb{C})$. This metric is bi-invariant on $G$ and the corresponding geodesics are the curves 
$\gamma(t)=P\exp({t X})$, 
where $X \in \mathfrak{g}$ and $P \in G$. The set of  minimizing geodesic segments of $(G,\phi)$ is a classical and relevant subject of investigation.

In this paper we also assume that the group $G$ is \emph{SVD-closed}: a condition satisfied  by many closed subgroup of $U_n$. The reason is that, under this assumption, for every $P_{_0}, P_{_1} \in G$, the set of  minimizing geodesic segments of $(G, \phi)$ with endpoints $P_{_0}$ and $P_{_1}$, can be parametrized by the set of generalized principal $\mathfrak{g}$--logarithms of $P_{_0}^* P_{_1}$ (see Theorem \ref{distanza}). 

Therefore, a geometric structure on $\mathfrak{g}$--$plog(P_{_0}^* P_{_1})$ induces a corresponding structure on the set of minimizing geodesic segments joining $P_{_0}$ and $P_{_1}$.

To fully illustrate the statements of the title and of the previous result, we must explain the meaning of \emph{SVD-closure}.
Any matrix $M \in \mathfrak{gl}_n (\mathbb{C})  \setminus \{0\}$ has a unique decomposition (called \emph{SVD-decomposition} of $M$) of the form $M = \sum\limits_{i=1}^p \sigma_{_i} A_{_i}$, where $\sigma_{_1} > \sigma_{_2} > \cdots > \sigma_{_p} > 0$ are the non-zero singular values of $M$, and $A_{_1}, A_{_2}, \cdots A_{_p}$ are non-zero complex matrices (called \emph{SVD-components} of $M$) such that $A_{_h}^* A_{_j} = A_{_h} A_{_j}^* =0$, for every $h \ne j$, and $A_{_j} A_{_j}^* A_{_j} = A_{_j}$, for every $j$. We say that a real Lie subalgebra $\mathfrak{g}$ of $\mathfrak{gl}_n(\mathbb{C})$ is \emph{SVD-closed} if, for any matrix $M \in \mathfrak{g} \setminus \{0\}$, all SVD-components of $M$ belong to $\mathfrak{g}$.
A closed subgroup of $GL_n (\mathbb{C})$ is \emph{SVD-closed} if its Lie algebra is SVD-closed in $\mathfrak{gl}_n(\mathbb{C})$. 

\smallskip

Sections \ref{preliminari} and \ref{preliminari2} are devoted to recall many general basic notions and preliminary facts on matrices. In Section \ref{SVD-closed} we discuss and determine a wide class of SVD-closed real Lie subalgebras of $\mathfrak{gl}_n (\mathbb{C})$. The key result is that the sets of fixed points of all automorphisms of the real Lie algebra $\mathfrak{gl}_n(\mathbb{C})$, commuting with the map $\eta: A \mapsto A^*$ and preserving the so-called \emph{triple Jordan product}, are SVD-closed real Lie subalgebras of $\mathfrak{gl}_n (\mathbb{C})$ (see Proposition \ref{Fix-SVD-inv}). In Section \ref{SVD-gruppi}, we prove that many classical groups of matrices are SVD-closed, as, for instance, the real general linear group $GL_n (\mathbb{R})$, the unitary group $U_n$, the special orthogonal complex group $SO_n (\mathbb{C})$, the symplectic groups
$Sp_{2n} (\mathbb{C})$, $Sp_{2n} (\mathbb{R})$, the generalized unitary groups $U_{(p, n-p)}$ and all their intersections.
In particular, we analyse the following families of SVD-closed subgroups of $U_n$:

$\langle  V  \rangle_{_{U_n}} := \{X \in U_n : XV=VX \}$, where $V$ is an arbitrary unitary matrix, 

$\preccurlyeq Q \succcurlyeq_{_{U_n}} :=\{X \in U_n : XQX^T=Q \}$ \ and \ 
$\preccurlyeq Q \succcurlyeq_{_{SU_n}}\  := \preccurlyeq Q \succcurlyeq_{_{U_n}} \cap \ SU_n$, \ where $Q$ is an arbitrary real orthogonal matrix. Among them, we find many classical closed subgroups of $U_n$, as, for instance,  \ \ $SO_n$, \ \  $Sp_n$, \ \ $U_{(p,n-p)} \cap U_n$ \  and \  $\big(SO_{(p, n-p)} (\mathbb{C})\big) \cap U_n$.

In Section \ref{Sect-plog} we study the set $\mathfrak{g}$--$plog(M)$ for a matrix $M$, belonging to a connected SVD-closed subgroup $G$ of $U_n$, with Lie algebra $\mathfrak{g}$. In particular we prove that  $\mathfrak{g}$--$plog(M)$ is non-empty (see Proposition \ref{esist-gplog}) and that it is a disjoint union of a finite number of compact submanifolds of $\mathfrak{g}$, each of which is diffeomorphic to a homogeneous space (Theorem \ref{maximal-torus}).
In Section \ref{Sect-Frob-metr} we obtain some results about of the Riemannian manifold $(G, \phi)$, where $G$ is any connected SVD-closed subgroup of $U_n$, and, among them, the already mentioned Theorem \ref{distanza}. In addition, we compute the diameter of all connected SVD-closed subgroups of $U_n$ that we considered in Section \ref{SVD-gruppi} (see Proposition \ref{diameter}).

The main result of Section \ref{Sect-U-n} is Theorem \ref{main-thm-par5}, in which we prove that, for every $V \in U_n$ and $M\in \langle V \rangle_{_{U_n}}$, the set
$\langle V \rangle_{_{\mathfrak{u}_n}}\!\!$--$plog(M)$ has a finite number of components, each of which is a simply connected compact submanifold of $\mathfrak{u}_n$, diffeomorphic to the product of suitable complex Grassmannians.
Finally, the main result of Section \ref{Sect-Q-in-O-n} is Theorem \ref{teor-compl}, which states that, for every  $Q \in O_n$ and $M \in \ \preccurlyeq  Q  \succcurlyeq_{_{SU_n}}$, the set $\preccurlyeq Q \succcurlyeq_{_{\mathfrak{s\!u}_n}}\!\!$--$plog(M)$ has a finite number of components, each of which is a simply connected compact submanifold of $\mathfrak{su}_n$, diffeomorphic to the product of suitable complex Grassmannians with the symmetric homogeneous spaces $\dfrac{SO_{2m}}{U_{m}}$ and $\dfrac{Sp_{\mu}}{U_{\mu}}$.

\section{Basic notations and some preliminary facts.}\label{preliminari}

\begin{notas}\label{notazioni} \ \\
a) In this paper we will use many standard notations from the matrix theory and from the theory of Lie groups and algebras.

Among these, if $\mathbb{K}$ is either the field of real numbers $\mathbb{R}$, or the field of complex numbers $\mathbb{C}$, or the associative division algebra of quaternions $\mathbb{H}$, then $\mathfrak{gl}_n (\mathbb{K})$ denotes the real Lie algebra of square matrices of order $n$ and $GL_n (\mathbb{K})$ the Lie group of invertible matrices of order $n$, both with coefficients in $\mathbb{K}$. In any case, the identity matrix and the null matrix of order $n$ are denoted by $I_{_n}$ and by ${\bf 0}_{_n}$, respectively, and we define also $\mathbb{K}^0= \lbrace0\rbrace\ .$ As usual, ${\bf i}$ is the unit imaginary number of $\mathbb{C}$ and ${\bf j}$, ${\bf k}$ are the further standard imaginary unities of $\mathbb{H}$, so that ${\bf i}^2 = {\bf j}^2 = {\bf k}^2 = -1$, ${\bf i}{\bf j} = - {\bf j}{\bf i} = {\bf k}$, ${\bf j}{\bf k} = - {\bf k}{\bf j} = {\bf i}$, ${\bf k}{\bf i} = - {\bf i}{\bf k} = {\bf j}$.  Note that any $q \in \mathbb{H}$ can be written in a unique way as $q= z + w {\bf j}$ with $z, w \in \mathbb{C}$, so that the complex field $\mathbb{C}$ can be identified with the set of quaternions of the form $z + 0 \! \cdot \! {\bf j}$, with $z \in \mathbb{C}$. We denote by $e^z := \sum\limits_{i=0}^{+ \infty} \dfrac{z^i}{i!}$ the exponential of $z \in \mathbb{C}$ and, if $z \ne 0$, by $\log(z)$, the unique complex logarithm of $z$, whose imaginary part lies in the interval $(- \pi , \pi]$.

For every $A \in \mathfrak{gl}_n(\mathbb{H})$, $A^T$, $\overline{A}$, $A^{*} := \overline{A}^T$ and $A^{-1}$ (provided that $A$ is invertible) are respectively transpose, conjugate, adjoint and inverse of the matrix $A$ and $tr(A)$ is its trace. If $A \in \mathfrak{gl}_n (\mathbb{C})$, $\det(A)$ denotes its determinant, while $\exp(A):= \sum\limits_{i=0}^{+ \infty} \dfrac{A^i}{i!} \in GL_n (\mathbb{C})$ denotes the exponential of the matrix $A$.

If $M_{_1}, \cdots , M_{_h}$ are square matrices of orders $r_{_1}, \cdots , r_{_h}$, respectively, then $M_{_1} \oplus \cdots \oplus M_{_h}$ denotes the related block-diagonal square matrix of order $r_{_1} + \cdots + r_{_h}$. Moreover, if $B$ is a $p \times p$ matrix, then $B^{\oplus h}$ denotes the $ph \times ph$ block-diagonal matrix $\underbrace{B \oplus \cdots \oplus B}_{h \mbox{ {\footnotesize times}}}$. 

\smallskip
If $\mathcal{S}_1, \dots , \mathcal{S}_m$ are sets of square matrices, then  $\mathcal{S}_1 \oplus \dots \oplus \mathcal{S}_m$ denotes the set of all matrices $B_1 \oplus \cdots \oplus B_m$ with $B_j \in \mathcal{S}_j$\ , for every $j$. \ If the sets $\mathcal{S}_1, \dots , \mathcal{S}_m$ are mutually disjoint, we write $\bigsqcup\limits_{i=1}^h S_{_i}$ \ to denote their (disjoint) union.

To give a full generality to the results of this paper (and to their proofs), it is necessary to establish agreements on the notations that we will use: if \ $h$ is a non-negative integer parameter, whenever, in any formula, we write any term as $\sum\limits_{i=1}^h  \ (\cdots), \ \ \bigoplus\limits_{i=1}^h \ (\cdots)$ \ or \ $\prod\limits_{i=1}^h \ (\cdots)$, \ we mean that, if $h=0$, this sum, this direct sum or this product must not appear in the related formula. 
Moreover, if $G_n$ \ (for $n \geq 1$) denotes any classical Lie groups of matrices of order $n$, having Lie algebra $\mathfrak{g}_n$, and if $H_n$ is a closed subgroup of $G_n$, we also assign a meaning to the expressions $G_0, \ \mathfrak{g}_0, \ \dfrac{G_0}{H_0}$, defining them all equal to a single point \ $\mathcal{Q}$ \ which, conventionally, satisfies the following conditions:

$\lambda \mathcal{Q} = \mathcal{Q}$, \ for every $\lambda \in \mathbb{C}$; \ \ \ \ \ \ \ $\mathcal{Q} \oplus B = B \oplus \mathcal{Q} = B$, \ for any square matrix $B$; \ \ \ \ \ \ \ $\mathcal{Q} \oplus \mathcal{S} = \mathcal{S} \oplus \mathcal{Q} = \mathcal{S}$, \ for any set of square matrices $\mathcal{S}$.

It is also useful to define the zero-order identity matrix $I_{_0}$ and 
$M^{\oplus 0}$ (for every square matrix $M$) both equal to this point $\mathcal{Q}$ and, to simplify the notations and some statements, the complex numbers, which are not eigenvalues of a matrix $M$, will be called \emph{eigenvalues of multiplicity zero} of $M$. 
Furthermore, we denote:

\smallskip

$\Omega:= \begin{pmatrix}
0 & -1 \\ 
1 & 0
\end{pmatrix} 
$;  
$\Omega_{_n} :=\begin{pmatrix}
{\bf 0}_{_n} & -I_{_n} \\ 
I_{_n} & {\bf 0}_{_n}
\end{pmatrix}$; 
hence $\Omega_{_1}= \Omega$,  while, for $n \ge 2$, we have  $\Omega_{_n} \ne \Omega^{\oplus n}$;

\smallskip

$W_{_{(p,q)}} := I_{p} \oplus {\bf i}I_{q}$, \ \ for every $ p, q  \geq 0$ such that $p+q \ge 1$  \ ($W_{_{(p,q)}}$ is unitary and diagonal);

$E_{_{\varphi}}:=
\begin{pmatrix}
\cos(\varphi) & -\sin(\varphi) \\ 
\sin(\varphi) & \cos(\varphi)
\end{pmatrix} = \cos(\varphi) \, I_2 + \sin(\varphi) \, \Omega
$, with $\varphi \in \mathbb{R}$, \ so \  $\Omega = E_{_{\pi/2}}$ \ and  \ $E_{_{\varphi}}^{\oplus h} = \cos(\varphi) I_{_{2h}} + \sin(\varphi) \Omega^{\oplus h}$ for every $h \ge 1$; \ \ \ 

moreover, for every $p, q \ge 0$ with $p+q \ge 1$, \ \ \ \ $E_{_{\varphi}}^{(p, q)} := E_{_{\varphi}}^{\oplus p} \oplus (-E_{_{\varphi}})^{\oplus q}$ (so $E_{_{\varphi}}^{(n, 0)}=  E_{_{\varphi}}^{\oplus n})$ \  and \ 
$J^{(p, q)}:= I_{_p} \oplus (-I_{_q}) = E_{_0}^{(p, q)}$ \ \ \ (so $J^{(p, 0)}= I_{_p}$ and $J^{(0, q)}= - I_{_q}$).

\smallskip

b) As usual, 
\ $O_n:=\{ X \in \mathfrak\mathfrak{gl}_n (\mathbb{R}) : XX^T = I_{_n}\}$ is the real orthogonal group; 

$U_n:=\{ X \in \mathfrak{gl}_n (\mathbb{C}) : XX^* = I_{_n}\}$ is the (complex) unitary group;

$SO_n :=\{ X \in O_n : \det(X) =1 \}$, $SU_n :=\{ X \in U_n : \det(X) =1 \}$ are their special subgroups;
\ while \ $U_n (\mathbb{H}):=\{ X \in \mathfrak{gl}_n (\mathbb{H}) : XX^* = I_{_n}\}$ is the quaternionic unitary group.

Note that the identification (recalled in (a)) of $\mathbb{C}$ as a subalgebra of $\mathbb{H}$, allows to identify $U_n$ with a subgroup of $U_n (\mathbb{H})$. In this paper this identification is always implied and not explicitly indicated.
\ \ Furthermore, for every $ p, q \ge 0$, with $p+q \ge 1$,

$O_{(p, q)} (\mathbb{C}):= \{X \in  \mathfrak{gl}_{(p+q)} (\mathbb{C}) : X J^{(p, q)} X^T = J^{(p, q)}\}$,

$SO_{(p, q)} (\mathbb{C}):= \{X \in  O_{(p, q)} (\mathbb{C}) : det(X)=1\}$,

$O_{(p, q)}:= O_{(p, q)} (\mathbb{C}) \cap \mathfrak{gl}_{(p+q)} (\mathbb{R})$, 
\ \ \ \ $SO_{(p, q)}:= SO_{(p, q)} (\mathbb{C}) \cap \mathfrak{gl}_{(p+q)} (\mathbb{R})$, 

are the complex and real  indefinite orthogonal groups, with their special subgroups;

$U_{(p, q)}:= \{X \in  \mathfrak{gl}_{(p+q)} (\mathbb{C}) : X J^{(p, q)} X^* = J^{(p, q)}\}$ 
is the indefinite unitary group.
Finally
$Sp_{2n} (\mathbb{C}) := \{ X \in \mathfrak{gl}_{2n} (\mathbb{C}) : X \Omega_{_n} X^T = \Omega_{_n}\}$ \ \ and \ \ 
$Sp_{2n} (\mathbb{R}):= Sp_{2n} (\mathbb{C}) \cap \mathfrak{gl}_{2n} (\mathbb{R})$ \ are,

respectively, the complex and real symplectic groups;
while \ $Sp_n := Sp_{2n} (\mathbb{C}) \cap U_{2n}$ \ is the compact symplectic group.
Of course, all the previous are real Lie groups of matrices. 

We recall that a well-known Cartan theorem states that a subgroup $H$ of a given Lie group $G$  is closed if and only if it is an embedded real submanifold of $G$. Of course, if the Lie group $G$ is compact, then every closed subgroup of $G$ is compact too.

If $G$ is any Lie group and $P \in G$, then $T_{_P} (G)$ denotes the tangent space of $G$ at $P$. 

\smallskip

c) The Lie algebras related to the previous Lie groups are denoted by:

$\mathfrak{so}_n = \{A \in \mathfrak\mathfrak{gl}_n (\mathbb{R}) : A = -A^T\}$, the Lie algebra of both $O_n$ and $SO_n$;

$\mathfrak{u}_{n}= \{A \in \mathfrak{gl}_n (\mathbb{C}) : A = -A^*\}$, the Lie algebra of $U_n$;

$\mathfrak{su}_n = \{A \in \mathfrak{gl}_n (\mathbb{C}) : A = -A^*, \ tr(A)=0 \}$, the Lie algebra of $SU_n$;

$\mathfrak{u}_n (\mathbb{H}) = \{A \in \mathfrak{gl}_n (\mathbb{H}) : A = -A^* \}$, the Lie algebra of $U_n (\mathbb{H})$.

The Lie algebras of the remaining Lie groups will be denoted by the corresponding small gothic letters: for instance, $\mathfrak{so}_{(p, q)} (\mathbb{C})$ and  $\mathfrak{sp}_n$ are the Lie algebras of $SO_{(p, q)} (\mathbb{C})$ and  of $Sp_n$, respectively.

\smallskip

d) If $B \in GL_n (\mathbb{C})$, we denote by $Ad_{_B}$ the map from $\mathfrak{gl}_n (\mathbb{C})$ onto itself, defined by 

$Ad_{_B}: A \mapsto Ad_{_B}(A):= BAB^{-1}$. 
Note that $Ad_{_B}$ commutes with the exponential map.
In this paper, we will still denote by  $Ad_{_B}$ the restriction of this map to any subset of $\mathfrak{gl}_n (\mathbb{C})$.
We indicate with $\tau$, $\mu$ and $\eta$ the  maps from $\mathfrak{gl}_n (\mathbb{C})$ onto itself, given by:
\ \ $\tau: A \mapsto A^T$,  \ \ \ \ \ \ $\mu: A \mapsto \overline{A}$, \ \ \ \ \ \ $\eta: A \mapsto A^{*}$. 
The maps $\mu$, $-\tau$, $-\eta$ and $Ad_{_B}$ (with $B \in GL_n (\mathbb{C})$) are automorphisms of the real Lie algebra $\mathfrak{gl}_n (\mathbb{C})$; furthermore, the automorphisms $\mu$, $-\tau$, $-\eta$ are involutive, mutually commuting and the composition of any two of them is the third automorphism; hence the group generated by \ $\mu$, $-\tau$, $-\eta$ \ is isomorphic to $\mathbb{Z}_{_2} \oplus \mathbb{Z}_{_2}$.

\smallskip

e) We denote by ${\phi}$  the \emph{Frobenius} (or \emph{Hilbert-Schmidt}) positive definite  real scalar product on $\mathfrak{gl}_n (\mathbb{C})$, defined by ${\phi}(A, B) := Re(tr(A B ^*))$, and we denote by $\Vert A \Vert_{_\phi} := \sqrt{{\phi}(A, A)} = \sqrt{tr(AA^*)}$, the related \emph{Frobenius norm}. Note that, if $A \in \mathfrak{u}_n$, then 
$\Vert A \Vert_{_\phi}^2 =  -tr(A^2)$. Since the eigenvalues of the skew-hermitian matrix $A$ are purely imaginary, we also get $\Vert A \Vert_{_\phi} = \sqrt{-tr(A^2)} = \sqrt{\sum\limits_{j=1}^n |\lambda_{_j}|^2}$, where $\lambda_{_1} , \cdots , \lambda_{_n}$ are the $n$ eigenvalues of $A$.
\end{notas}

\begin{rems}\label{identificazioni} 
a) The map
$\rho: \mathbb{C} \to \mathfrak{gl}_2 (\mathbb{R})
$, given by $\rho(z) :=  Re(z) I_{_2} + Im(z) \Omega= \begin{pmatrix}
Re(z) & -Im(z) \\ 
Im(z) & Re(z)
\end{pmatrix} 
$,
 is a monomorphism of $\mathbb{R}$-algebras, such that
 $\rho({\overline{z}}) = \rho(z)^T$
and such that $\rho(z) \in GL_2 (\mathbb{R})$ as soon as $z \ne 0$.
More generally, for any $h \ge 1$, we denote again by $\rho$ the mapping: $\mathfrak{gl}_h (\mathbb{C}) \to \mathfrak{gl}_{2h} (\mathbb{R})$, which maps the $h \times h$ complex matrix $Z=(z_{ij})$ to the block matrix  $\rho(Z)=(\rho(z_{ij})) \in \mathfrak{gl}_{2h} (\mathbb{R})$, having $h^2$ blocks of order $2 \times 2$. We say that $\rho$ is the \emph{decomplexification} map.
It is not hard to prove that, if $\lambda_1, \cdots, \lambda_h$ are the $h$ eigenvalues of any matrix $Z \in \mathfrak{gl}_h (\mathbb{C})$, then $\lambda_1, \overline{\lambda}_1, \cdots, \lambda_h, \overline{\lambda}_h$ are the $2h$ eigenvalues of $\rho(Z) \in  \mathfrak{gl}_{2h} (\mathbb{R})$ and that $\rho$ is a monomorphism of $\mathbb{R}$-algebras, whose restriction to $GL_h (\mathbb{C})$ is a monomorphism of Lie groups, having as image
$\rho \big(\mathfrak{gl}_h (\mathbb{C}) \big) \cap GL_{2h} (\mathbb{R})$.
We have also $\rho(Z^*) = \rho(Z)^T$; so, the restriction of $\rho$ to $U_h$ is a monomorphism of Lie groups
and 
$\rho \big( U_h \big)=  \rho \big( \mathfrak{gl}_h (\mathbb{C}) \big) \cap SO_{2h}$.
From now on, to simplify the notations, the map $\rho$ will be omitted, hence we will regard  the real Lie algebra $\mathfrak{gl}_h (\mathbb{C})$ as Lie subalgebra of $\mathfrak{gl}_{2h} (\mathbb{R})$, the Lie groups $GL_h (\mathbb{C})$ and $U_h$ as closed subgroups of $GL_{2h} (\mathbb{R})$ and $SO_{2h}$, respectively; in particular we will write $U_h =   \mathfrak{gl}_h (\mathbb{C})  \cap SO_{2h}$.

\smallskip

b) We denote by $\Psi : \mathbb{H} \to \mathfrak{gl}_2 (\mathbb{C})$ the map: $z + w {\bf j}\mapsto  \Psi(z + w {\bf j}) :=\begin{pmatrix}
z & -w \\ 
\overline{w} & \overline{z}
\end{pmatrix}$, where $z, w \in \mathbb{C}$\ ; this map is a monomorphism of $\mathbb{R}$-algebras. Note that, for every $q \in \mathbb{H}$, we have  $\Psi(\overline{q}) = (\Psi(q))^*$.
It is possible to extend this map to a monomorphism of $\mathbb{R}$-algebras (still denoted by the same symbol)  
$\Psi: \mathfrak{gl}_h (\mathbb{H}) \to \mathfrak{gl}_{2h} (\mathbb{C})$ \ ($h\geq1$), which maps the $h \times h$ quaternion matrix $Q=(q_{_{ij}})$ to the block matrix $\Psi(Q)=\big( \Psi(q_{_{ij}})\big) \in \mathfrak{gl}_{2h} (\mathbb{C})$, having $h^2$ blocks of order $2 \times 2$.
It can be easily checked that we have $\Psi(A^*) = (\Psi(A))^*$ and $(\Omega^{\oplus h}) \Psi(A^*) (\Omega^{\oplus h})^T = (\Psi(A))^T$, for every $A \in \mathfrak{gl}_h (\mathbb{H})$. Moreover, $\Psi$ maps $GL_h (\mathbb{H})$ into $GL_{2h} (\mathbb{C})$ and $U_h (\mathbb{H})$ into $U_{2h}$; both restrictions $GL_h (\mathbb{H}) \to GL_{2h} (\mathbb{C})$ and  $U_h (\mathbb{H}) \to U_{2h}$ are monomorphisms of Lie groups. 
Hence, up to the isomorphim $\Psi$, we will consider $\mathfrak{gl}_h (\mathbb{H})$ as real Lie subalgebra of $\mathfrak{gl}_{2h} (\mathbb{C})$, \ \ $GL_h (\mathbb{H})$ as closed subgroup of $GL_{2h} (\mathbb{C})$ \ and \ $U_h (\mathbb{H})$ as closed subgroup of $U_{2h}$.

Note also that the monomorphism $\Psi$ maps the closed subgroup $U_h$ of $U_h (\mathbb{H})$ onto a closed subgroup of $\Psi(U_h (\mathbb{H})) \subset U_{2h}$, so that the elements of $\Psi(U_h)$ are  the $2h \times 2h$ complex unitary matrices, having $h^2$ blocks $Z_{_{ij}}$ of the form: $Z_{_{ij}} = \begin{pmatrix}
z_{_{ij}} & 0 \\ 
0 & \overline{z}_{_{ij}}
\end{pmatrix}$, with $z_{_{ij}} \in \mathbb{C}$.

As in the case of the map $\rho$, from now on, to simplify the notations, we will omit to indicate the map $\Psi$ and so, for instance, we will simply write $U_h (\mathbb{H})=U_{2h} \cap \mathfrak{gl}_h (\mathbb{H})$ and $\mathfrak{u}_h (\mathbb{H})=\mathfrak{u}_{2h} \cap \mathfrak{gl}_h (\mathbb{H})$. From this last equality, we easily get that  every matrix of $\mathfrak{u}_h (\mathbb{H})$ has trace $0$. Therefore, since $U_h (\mathbb{H}) = \exp\big(\mathfrak{u}_h (\mathbb{H})\big)$, the  group $U_h (\mathbb{H})$ is contained in  $SU_{2h}$, hence \ $U_h (\mathbb{H})=SU_{2h} \cap \mathfrak{gl}_h (\mathbb{H})$ \ and \ $\mathfrak{u}_h (\mathbb{H})=\mathfrak{su}_{2h} \cap \mathfrak{gl}_h (\mathbb{H})$.

\smallskip

c) Fixed $n \ge 1$, for  any $i, j = 1 , \cdots , 2n$, let $W(i,j)$ be the square matrix of order $2n$, having $1$ at the entry $(i, j)$ and $0$ elsewhere, and let $B$ be the $2n \times 2n$ real matrix defined by $B := \sum\limits_{j=1}^n \big( W( j, 2j-1) + W( n+j, 2j) \big)$. Since $W(i, j) W(h,k) = \delta_{_{jh}} W(i,k)$, it is easy to check that $B$ is an orthogonal matrix such that $B^T \Omega_{_n}B = \Omega^{\oplus n}$\ ; from this, one can get that $X$ belongs to $U_n (\mathbb{H})$ if and only if $B\,X B^T$ belongs to $Sp_n$\ , i.e. $Ad_{_B}\big(U_n (\mathbb{H})\big) = Sp_n$. It is also easy to check that $Ad_{_B}$ maps the closed subgroup $U_n$ of $U_n (\mathbb{H})$ onto the closed subgroup of $Sp_n$ of matrices of the form $A \oplus \overline{A}$ with $A \in U_n$. Hence $U_n$ can be regarded as the closed subgroup of $Sp_n$ of matrices of this form, and so, the simply connected compact symmetric homogeneous space $\dfrac{Sp_n}{U_n}$, obtained in this way , is diffeomorphic to $\dfrac{U_n (\mathbb{H})}{U_n}$.

\smallskip

d) Let $\Phi$ be the automorphism of $\mathbb{R}$-algebra $\mathbb{H}$, defined by $\Phi(t+ x {\bf i} +  y {\bf j} + z {\bf k} )= t+ y {\bf i} + x {\bf j} - z {\bf k} $, for every $t, x, y, z \in \mathbb{R}$. We have: $\Phi(\overline{q})=\overline{\Phi(q)}$, for every $q \in \mathbb{H}$. Acting on each single entry of the matrix, this map induces an automorphism (still denoted by $\Phi$) of the $\mathbb{R}$-algebra $\mathfrak{gl}_n (\mathbb{H})$. Since $ \Phi(A^*) = \Phi(A)^* $, for every $A \in \mathfrak{gl}_n (\mathbb{H})$, the restriction of $\Phi$ to $U_n (\mathbb{H})$ is an automorphism of Lie group $U_n (\mathbb{H})$, which maps $U_n$ onto a closed subgroup of $U_n (\mathbb{H})$. Hence the homogeneous space $\dfrac{U_n (\mathbb{H})}{\Phi(U_n)}$ is diffeomorphic to $\dfrac{U_n (\mathbb{H})}{U_n}$ and, by (c), also to $\dfrac{Sp_n}{U_n}$.

Remembering (b), up to the map $\Psi$, the subgroup $\Phi(U_n)$ of $U_n (\mathbb{H})$ can be identified with the subgroup of $U_{2n}$, whose elements are  the $2n \times 2n$ special orthogonal matrices, having $n^2$ real blocks $U_{_{ij}}$ of the form: $U_{_{ij}} = \begin{pmatrix}
x_{_{ij}} & - y_{_{ij}} \\ 
y_{_{ij}} & x_{_{ij}}
\end{pmatrix}$. Note that, remembering (a), the restriction of $\Phi$ to $U_n$  agrees with the restriction to $U_n$ of the decomplexification map $\rho$.
\end{rems}

\section{Commuting matrices and SVD-systems}\label{preliminari2}

\begin{nota}\label{parentesi-angolari}
Let $\mathcal{S} \subseteq \mathfrak{gl}_n (\mathbb{C})$ and  $M \in \mathfrak{gl}_n (\mathbb{C})$.  We denote

$\langle  M \rangle_{_{\mathcal{S}}} := \{ X \in \mathcal{S} : XM = MX \}$ \ \ \ \ \ \  and \ \ \ \ \ 
$\preccurlyeq M \succcurlyeq_{_{\mathcal{S}}} := \{ X \in \mathcal{S} : XM = M\overline{X} \}$.
\end{nota}

\begin{rems}\label{parentesi-angolari2}
a) Let $A \in U_n$ , $M \in \mathfrak{gl}_n (\mathbb{C})$ and  $\mathcal{S} \subseteq \mathfrak{gl}_n (\mathbb{C})$. It is easy to check that $Ad_{_A}\big(\preccurlyeq M \succcurlyeq_{_{\mathcal{S}}}\big) = \ \preccurlyeq AMA^T \succcurlyeq_{_{Ad_{_{\!A}}\!(\mathcal{S})}} .$ 

In particular, if $A \in O_n$, we get $Ad_{_A}\big(\preccurlyeq M \succcurlyeq_{_{\mathcal{S}}}\big) = \ \preccurlyeq Ad_{_{A}}(M) \succcurlyeq_{_{Ad_{_{\!A}}\!(\mathcal{S})}} .$

b) Let $G$ be a closed subgroup of $GL_n (\mathbb{C})$, having $\mathfrak{g} \subseteq \mathfrak{gl}_n (\mathbb{C})$ as Lie algebra and let M be any matrix in $\mathfrak{gl}_n (\mathbb{C})$. Then $\langle  M \rangle_{_G}$  and $\preccurlyeq M \succcurlyeq_{_G}$ are closed subgroups of $G$, whose Lie algebras are $\langle  M \rangle_{_{\mathfrak{g}}}$  and $\preccurlyeq M \succcurlyeq_{_{\mathfrak{g}}}$, respectively.
\end{rems}

\begin{lemma}\label{commutare-phi} a) Let $\varphi \in \mathbb{R}$, $\varphi \ne k \pi$, $k \in \mathbb{Z}$. Any matrix of $\mathfrak{gl}_{2n} (\mathbb{C})$ commutes with $E_{_{\varphi}}^{\oplus n}$ if and only if it commutes with $\Omega^{\oplus n}$, i.e.  $\langle E_{_{\varphi}}^{\oplus n} \rangle_{_{\mathfrak{g\!l}_{2n}\!(\mathbb{C})}} = \langle \Omega^{\oplus n} \rangle_{_{\mathfrak{g\!l}_{2n}\!(\mathbb{C})}}\ .$

b) Let $\mathcal{S}$ be any subset of $\mathfrak{gl}_{2n} (\mathbb{C})$, then $\langle \Omega^{\oplus n} \rangle_{_{\mathcal{S}}}$ consists of the matrices of $\mathcal{S}$, having $n^2$ blocks of the form: $X_{_{ij}} = 
\begin{pmatrix}
a_{_{ij}} & -b_{_{ij}} \\ 
b_{_{ij}} & a_{_{ij}}
\end{pmatrix} 
$, with $a_{_{ij}}, b_{_{ij}} \in \mathbb{C}$.
\end{lemma}

\begin{proof}
Part (a) is trivial and follows from $E_{_{\varphi}}^{\oplus n} = \cos(\varphi) I_{_{2n}} + \sin(\varphi) \Omega^{\oplus n}$ and $\sin(\varphi)\neq 0$.

For part (b), we can write an arbitrary matrix of $\mathcal{S}$ in $n^2$ blocks, $X_{_{ij}}$, each of them of order $2$.  We easily get that such a matrix commutes with $\Omega^{\oplus n}$ if and only if each block commutes with $\Omega$, i. e. if and only if each $X_{_{ij}}$ is of the form stated in (b).
\end{proof}

\begin{lemma}\label{anticomm}
Let $D:= \bigoplus\limits_{j=1}^s  D_{_j} \in \mathfrak{gl}_n (\mathbb{C})$ be a block diagonal matrix, with 
$D_{_j} \in \mathfrak{gl}_{n_{_j}} (\mathbb{C})$ simisimple matrices.  Denote by $S_{_j}$ and by $-S_{_j}$ ($j = 1, \cdots , s$), respectively, the set of the eigenvalues of $D_{_j}$ and the sets of their opposites.

a) Assume that $S_{_i} \cap (-S_{_j}) = \emptyset$ as soon as $i \ne j$.
Then a matrix $A \in \mathfrak{gl}_n (\mathbb{C})$ anticommutes with $D$ if and only if $A = \bigoplus\limits_{j=1}^s  A_{_j}$, where each $A_{_j}$ belongs to $\mathfrak{gl}_{n_{_j}} (\mathbb{C})$ and anticommutes with $D_{_j}$.

b) Assume that $S_{_i} \cap S_{_j} = \emptyset$ as soon as $i \ne j$.
Then a matrix $A \in \mathfrak{gl}_n (\mathbb{C})$ commutes with $D$ if and only if $A = \bigoplus\limits_{j=1}^s  A_{_j}$, where each $A_{_j}$ belongs to $\mathfrak{gl}_{n_{_j}} (\mathbb{C})$ and commutes with $D_{_j}$.
\end{lemma}

\begin{proof}
We proof only part (a), being part (b) similar and easier.

We write the matrix $A$ in blocks $A = (A_{_{ij}})$, consistent with the block structure of $D$, so the condition $A D = - D A$ is equivalent to $A_{_{ij}} D_{_j} = -D_{_i} A_{_{ij}}$, for $i, j = 1 , \ \cdots , n$. Assume $i \ne j$ and let $\mathcal{B}$ be a basis of $\mathbb{C}^{n_{_j}}$, consisting of eigenvectors of $D_{_j}$. If $v \in \mathcal{B}$, with associated eigenvalue $\lambda$, then $D_{_i} (A_{_{ij}} v) = - A_{_{ij}} D_{_j} v = - \lambda (A_{_{ij}} v)$. This implies that $A_{_{ij}} v =0$, otherwise (against the assumptions made) $-\lambda$ would be eigenvalue of $D_i$. This holds for every $v \in \mathcal{B}$ and so, $A_{_{ij}}= {\bf 0}$, as soon as $i \ne j$. Therefore $A= \bigoplus\limits_{j=1}^s A_{_{jj}}$, where each $A_{_{jj}}$ anticommutes with $D_{_j}$.
The converse is trivial.
\end{proof}

\begin{remdef}\label{Ad-orbit}
If $M \in \mathfrak{gl}_n (\mathbb{C})$ and $G$ is a closed subgroup of  $GL_n (\mathbb{C})$, we call $Ad(G)$-\emph{orbit} of $M$, denoted by $Ad\big(G\big)(M)$, the set $\{Ad_{_B}(M)=BMB^{-1} : B \in G\}$. 

It is well-known that each orbit $Ad\big(G\big)(M)$ is an immersed submanifold of $\mathfrak{gl}_n (\mathbb{C})$, diffeomorphic to the homogeneous space $\dfrac{G}{\langle M \rangle_{_G}}$, being $\langle M \rangle_{_G}$ the isotropy subgroup of $M$ with respect to the action of $G$; furthermore, if $G$ is compact, then $Ad\big(G\big)(M)$ is a compact (embedded) submanifold of $\mathfrak{gl}_n (\mathbb{C})$
(see, for instance, \cite{EoM-Orbit}).
\end{remdef}


\begin{remsdefs}\label{SVD-rec}
A non-empty family of matrices $A_{_1}, \cdots , A_{_p} \in \mathfrak{gl}_n (\mathbb{C}) \setminus \{0\}$ is said to
be an \emph{SVD-system}, if
\ \ $A_{_h}^* A_{_j} = A_{_h} A_{_j}^* =0$, \ \ for every $h \ne j$, \ \ and
\ \ $A_{_j} A_{_j}^* A_{_j} = A_{_j}$,\ \  for every $j=1, \cdots, p$\ .
\ Note that, if $A_{_1}, \cdots , A_{_p}$ is an SVD-system, then

a) the matrices $A_{_1}, \cdots , A_{_p}$ are linearly independent over $\mathbb{C}$;

b) $c_{_1} A_{_1}, c_{_2}A_{_2},  \cdots , c_{_p}A_{_p}$ is still an SVD-system, \ if $c_{_j} \in \mathbb{C}$ \ and \ $|c_{_j}| = 1$, \ for  $j=1 , \cdots, p$.

We call \emph{SVD-decomposition} of $M \in \mathfrak{gl}_n (\mathbb{C}) \setminus \{0\}$, any decomposition
$M = \sum\limits_{j=1}^p \sigma_{_j} A_{_j}$,
where $A_{_1}, \cdots , A_{_p} \in \mathfrak{gl}_n (\mathbb{C}) \setminus \{0\}$ form an SVD-system and $\sigma_{_1} > \sigma_{_2} > \dots > \sigma_{_p} >0$ are positive real numbers.
\ Any matrix $M \in \mathfrak{gl}_n (\mathbb{C})  \setminus \{0\}$ has an SVD-decomposition $M = \sum\limits_{j=1}^p \sigma_{_j} A_{_j}$ and this decomposition is unique, i.e. if $M = \sum\limits_{h=1}^q \tau_{_h} B_{_h}$ is another SVD-decomposition, then $p=q$, $\sigma_{_j} = \tau_{_j}$ and $A_{_j}=B_{_j}$ for every $j=1, \cdots, p$.
The positive numbers $\sigma_{_1} , \sigma_{_2} , \dots , \sigma_{_p}$ are the distinct square roots of the non-zero eigenvalues of $M^*M$; they are known as the \emph{non-zero singular values} of $M$. We say that the matrices $A_{_1} , \cdots A_{_p}$ are the 
\emph{SVD-components} of $M$. For more information, see for instance \cite[Thm.\,2.6.3]{HoJ2013}, \cite[Thm.3.4]{OttPaol2015} and also \cite[\S\,4]{DoPe2017}.
\end{remsdefs}

\begin{lemma}\label{autov-SVD-A}
Let $A_{_1}, \cdots , A_{_p}$ be an SVD-system of skew-hermitian matrices of order $n$, let \ $\theta_{_1} > \theta_{_2} > \cdots > \theta_{_p}$ be real numbers and denote $M:= \sum\limits_{j=1}^p \theta_{_j} A_{_j}$. Then

a) the eigenvalues of $A_{_j}$ are: \ ${\bf i}$ \ with multiplicity \ $\mu_{_j} \ge 0$, \ ${-\bf i}$ \ with multiplicity \ $\nu_{_j} \ge 0$ (where $\mu_{_j}+ \nu_{_j} \ge 1$) and \ $0$ \ with multiplicity \ $n- (\mu_{_j}+\nu_{_j}) \ge 0$, \ for every $j=1, \cdots, p$;

b) the distinct eigenvalues of $M$ are ${\bf i}\theta_{_j}$ with multiplicity $\mu_{_j} \ge 0$, ${-\bf i}\theta_{_j}$ with multiplicity $\nu_{_j} \ge 0$ (for $j=1, \cdots , p$ and $\sum\limits_{j=1}^p(\mu_{_j}+\nu_{_j}) \ge p$), and $0$ with multiplicity $n- \sum\limits_{j=1}^p(\mu_{_j}+\nu_{_j}) \ge 0$.
\end{lemma}

\begin{proof}
Since $A_{_1}, \cdots , A_{_p}$ is an SVD-system of skew-hermitian matrices, each matrix $A_{_j}$ satisfies the matrix equation $X^3 + X =0$. This allows to obtain (a).

We have $A_{_h} A_{_j} =-A_{_h} A_{_j}^*=0$, for every $h \ne j$; these conditions imply that, if $v$ is an eigenvector of $A_{_j}$ associated with the eigenvalue $\bf i$ or $-\bf i$, then $A_{_h}v=0$, for every $j \ne h$. Moreover the same conditions give, in particular, that the matrices $A_{_h}$ and $A_{_j}$ commute, hence $A_{_1}, \cdots , A_{_p}$ are simultaneously diagonalizable (together with $M$) by means of a unitary matrix (see for instance \cite[Thm.\,2.5.5 p.\,135]{HoJ2013}).
Using a common (orthonormal) basis of eigenvectors, we easily obtain (b).
\end{proof}

\begin{lemma}\label{exp-svd}
Let $A_{_1}, A_{_2}, \cdots , A_{_p}$ be an SVD-system of skew-hermitian matrices of order $n$ \ and let \ $\alpha_{_1}, \alpha_{_2}, \cdots, \alpha_{_p}$ be complex numbers. Then
\begin{center}
$
\exp(\sum\limits_{j=1}^p \alpha_{_j} A_{_j}) = I_{_n} + \sum\limits_{j=1}^p \big[\sin (\alpha_{_j}) A_{_j} +(1-\cos (\alpha_{_j})) A_{_j}^2\big].
$
\end{center}

\end{lemma}

\begin{proof}
Since $A_{_1}, A_{_2}, \cdots , A_{_p}$ are skew-hermitian, as in the proof of Lemma \ref{autov-SVD-A}, the properties of being an SVD-system give: $A_{_h} A_{_j} =0$, for $h \ne j$ (so $A_{_h}$ and $A_{_j}$ commute), and $A_{_j}^3 = -A_{_j}$, for every $j$. Hence $(\alpha_{_j} A_{_j})^{2k-1} = (-1)^{k-1} \alpha_{_j}^{2k-1} A_{_j}$ and $(\alpha_{_j} A_{_j})^{2k} = (-1)^{k-1} \alpha_{_j}^{2k} A_{_j}^2$, for every $j =1 , \cdots ,p$ \ and for every $k \ge 1$. 
Therefore: \ \  
$\exp(\sum\limits_{j=1}^p \alpha_{_j} A_{_j}) = \prod\limits_{j=1}^p \exp(\alpha_{_j} A_{_j})=
 \prod\limits_{j=1}^p \big[I_{_n} + \sin (\alpha_{_j}) A_{_j} + (1-\cos (\alpha_{_j}))  A_{_j}^2\big] = I_{_n} + \sum\limits_{j=1}^p \big[\sin (\alpha_{_j}) A_{_j} +(1-\cos (\alpha_{_j})) A_{_j}^2\big] $.
\end{proof}

\begin{rem}\label{Rodrigues}
Lemma \ref{exp-svd} gives one of the possible generalizations of the classical Rodrigues' formula (see \cite[Thm.\,2.2]{GaXu2002} and \cite[Ex.\,4.11]{DoPe2018b}).  
Note also that, from this Lemma, we obtain \ $\exp(\alpha \Omega) = E_{_\alpha}$, \ for every $\alpha \in \mathbb{R}$.
\end{rem}

\section{SVD-closed real Lie subalgebras of  $\mathfrak{gl}_n (\mathbb{C})$}\label{SVD-closed}

\begin{remdef}\label{intersec-SVD}
We say that a real Lie subalgebra $\mathfrak{g}$ of $\mathfrak{gl}_n (\mathbb{C})$ is \emph{SVD-closed} if all SVD-components of every matrix of \ $\mathfrak{g} \setminus \{0\}$ \ belong to \ $\mathfrak{g}$.

Note that any intersection of SVD-closed real Lie subalgebras of $\mathfrak{gl}_n (\mathbb{C})$ is an SVD-closed real Lie subalgebra of $\mathfrak{gl}_n (\mathbb{C})$.
\end{remdef}

\begin{nota}\label{gruppoF} We denote by $\mathfrak{A}_{_n}$ the group, whose elements are the automorphisms $f$ of the real Lie algebra $\mathfrak{gl}_n (\mathbb{C})$, such that

i) $f \circ \eta = \eta \circ f$ \ \ \  (i.e. $f(A^*)=f(A)^*$, \ \ for every $A \in \mathfrak{gl}_n (\mathbb{C})$);

ii) $f(ABA) = f(A) f(B) f(A)$, \ \ for every $A, B \in \mathfrak{gl}_n (\mathbb{C})$ \ \ \ (i.e. $f$ preserves the so-called \emph{Jordan triple product}).
\end{nota}

\begin{lemma}\label{tipi-Fn}
The elements of $\mathfrak{A}_{_n}$ are precisely the following maps: 

(1) $X \mapsto Ad_{_V}(X)= VXV^*$,  \ \ \ \ (2) $X \mapsto \big( Ad_{_V} \circ \mu \big) (X) = V \overline{X} V^*$, 

(3) $X \mapsto \big( Ad_{_V} \circ (-\tau) \big) (X) = - VX^T V^*$, \  \ \ \ (4) $X \mapsto \big( Ad_{_V} \circ (-\eta) \big)(X) = - VX^* V^*$,

for every \ $V \in U_n$.
\end{lemma}

\begin{proof}
It is easy to check that the previous maps are elements of $\mathfrak{A}_{_n}$.

For the converse, consider the decomposition $\mathfrak{gl}_n (\mathbb{C}) = \mathcal{H}_n \oplus \mathfrak{u}_n$, where $\mathcal{H}_n$ is the real vector subspace of $\mathfrak{gl}_n (\mathbb{C})$ of hermitian matrices, so that every matrix $Z \in \mathfrak{gl}_n (\mathbb{C})$ can be uniquely written as $Z = \dfrac{Z + Z^*}{2} + \dfrac{Z-Z^*}{2}$,  with $\dfrac{Z + Z^*}{2} \in \mathcal{H}_n$ and $\dfrac{Z-Z^*}{2} \in \mathfrak{u}_n$; let $f \in \mathfrak{A}_{_n}$ and denote by $f_{_1}$ and by $f_{_2}$ the restrictions of $f$ to $\mathcal{H}_n$ and to $\mathfrak{u}_n$, respectively. Since $f \circ \eta = \eta \circ f$, we have $f_{_1} (\mathcal{H}_n) = \mathcal{H}_n$ and  $f_{_2} (\mathfrak{u}_n) = \mathfrak{u}_n$. By \cite[Thm.\,2.1]{AH2006}, there exists a unitary matrix $V \in U_n$ such that we have 

either \ \ $f_{_1} =  Ad_{_V}$ \ \  or \ \  
$f_{_1} =  - Ad_{_V}$ \ \  or \ \   
$f_{_1} =   Ad_{_V} \circ \mu$ \ \ or \ \  
$f_{_1} =   - Ad_{_V} \circ \mu$.

In particular, this implies $f(I_{_n}) = \pm I_{_n}$.

Now we denote $\mathcal{M}:= {\bf i} I_{_n}$  and $\mathcal{N} := I_{_n} - \mathcal{M} =(1-{\bf i}) I_{_n}$, so that $\mathcal{N} Y \mathcal{N} = - 2 {\bf i} Y$, for every $Y \in \mathfrak{gl}_n (\mathbb{C})$.  
Since  $f$ is an automorphism of the Lie algebra $\mathfrak{gl}_n (\mathbb{C})$ and $\mathcal{M}$ belongs to its center $\mathcal{Z}$, then also $f(\mathcal{M})$ belongs to $\mathcal{Z}$, i.e. $f(\mathcal{M}) = \lambda I_{_n}$ for some $\lambda \in \mathbb{C}$. 
Since $f$ preserves the Jordan triple product, we get: $-f(I_{_n}) = f(\mathcal{M} I_{_n}\mathcal{M}) = \lambda^2 f(I_{_n})$. Hence $\lambda = \pm {\bf i}$, so that $f(\mathcal{N})= f(I_{_n})- f(\mathcal{M}) = (\varepsilon_{_1}  + \varepsilon_{_2} {\bf i} ) I_{_n}$, where $\varepsilon_{_1}, \varepsilon_{_2} = \pm 1$; from this we get $f(\mathcal{N})^2=  2 \varepsilon {\bf i}I_{_n}$, where $\varepsilon = \pm 1$. Fixed $Y \in \mathfrak{u}_n$, we have $({\bf i} Y)^* = {\bf i} Y$  and, so, $\mathcal{N} Y \mathcal{N} = - 2 {\bf i} Y \in \mathcal{H}_n$. Remembering that $f$ preserves the Jordan triple product, we get $-2f_{_1}({\bf i} Y) = f_{_1}(\mathcal{N} Y \mathcal{N}) = f(\mathcal{N}) f_{_2}( Y) f(\mathcal{N}) =  2 \varepsilon {\bf i}f_{_2}( Y)$ and this gives $f_{_2}(Y) = \varepsilon {\bf i} f_{_1} ({\bf i} Y)$. 
This last equality implies that $f(Z) = \dfrac{1}{2} \big[ f_{_1}(Z+Z^*) + \varepsilon {\bf i}  f_{_1}({\bf i}Z-{\bf i}Z^*) \big]$, for every $Z \in \mathfrak{gl}_n (\mathbb{C})$. Taking into account the four possible expressions for $f_{_1}$ (and the fact that $\varepsilon = \pm 1$),  easy computations allow to obtain the following eight possible expressions for $f$:

$\pm Ad_{_V}$,  \ \ \ \ $\pm Ad_{_V} \circ \mu$,
\ \ \ \ $\pm Ad_{_V} \circ \eta$,  \ \ \ \ $\pm Ad_{_V} \circ \tau$.

But \ $- Ad_{_V}$,  \ $-  Ad_{_V} \circ \mu$, \ $Ad_{_V} \circ \eta$, \ $ Ad_{_V} \circ \tau$ \ are  not automorphisms of the real Lie algebra $\mathfrak{gl}_n (\mathbb{C})$, while the remaining four are the expressions for $f$ in the statement.
\end{proof}

\begin{rem}\label{comm-anticomm}
If $f \in \mathfrak{A}_{_n}$, then either $f(XY) = f(X)f(Y)$ for every $X, Y \in \mathfrak{gl}_n (\mathbb{C})$ (in the cases (1) and (2) of Lemma \ref{tipi-Fn}) or  $f(XY) = -f(Y)f(X)$ for every $X, Y \in \mathfrak{gl}_n (\mathbb{C})$ (in the remaining cases (3) and (4)).
\end{rem}

\begin{prop}\label{Fix-SVD-inv}
For every $f \in \mathfrak{A}_{_n}$, \  
the set $Fix(f) := \{M \in  \mathfrak{gl}_n (\mathbb{C}) : f(M) = M \}$ 
is an SVD-closed real Lie subalgebra of $\mathfrak{gl}_n (\mathbb{C})$.
\end{prop}

\begin{proof}
Choose an element $f$ of  $\mathfrak{A}_{_n}$;  $Fix(f)$ is a real Lie subalgebra of $\mathfrak{gl}_n (\mathbb{C})$, since $f$ is an automorphism of the real Lie algebra $\mathfrak{gl}_n (\mathbb{C})$. Hence it suffices to prove that $Fix(f)$ is SVD-closed.
Let $M = \sum\limits_{i=1}^p \sigma_{_i} A_{_i}$ be a matrix of $Fix(f) \setminus \lbrace0\rbrace$, with its SVD-decomposition; since $f$ is $\mathbb{R}$-linear, we have 
$M=f(M) = \sum\limits_{i=1}^p \sigma_{_i} f(A_{_i})$.
By conditions (i), (ii) of Notation \ref{gruppoF}, we have $f(A_{_i}) f(A_{_i})^* f(A_{_i})= f(A_{_i} A_{_i}^* A_{_i}) = f(A_{_i})$, for $i= 1, \cdots , p$. Furthermore, by Remark \ref{comm-anticomm}, $f(A_{_i}) f(A_{_j})^*$ equals either $f(A_{_i} A_{_j}^*)$ or  $- f(A_{_j}^* A_{_i})$ and, in both cases, $f(A_{_i}) f(A_{_j})^* =0$, if $i\ne j$. Similarly, we get $f(A_{_i})^* f(A_{_j})=0$, if $i \ne j$.
Hence $\sum\limits_{i=1}^p \sigma_{_i} f(A_{_i})$ is another SVD-decomposition of $M$; by uniqueness, we get $f(A_{_i})=A_{_i}$, so   every $A_{_i} \in Fix(f)$.
\end{proof}

\begin{examples}\label{esempi}
From Proposition \ref{Fix-SVD-inv} and from Lemma \ref{tipi-Fn}, we obtain that, \ for every $V \in U_n$,  the following are SVD-closed real Lie subalgebras of $\mathfrak{gl}_n (\mathbb{C})$:

$Fix\big(Ad_{_V}\big) = \langle  V  \rangle_{_{\mathfrak{g\!l}_n\!(\mathbb{C})}}$; \ \ \ \ \ \ \ \ $Fix\big(Ad_{_V}\circ \mu\big) =  \preccurlyeq V \succcurlyeq_{_{\mathfrak{g\!l}_n\!(\mathbb{C})}}$;
\ \ \ \ \ \ \ \ $ Fix\big(Ad_{_V}\circ (-\tau)\big)$;

$ Fix\big(Ad_{_V}\circ (-\eta)\big)$ \ \ \ \ \ (note that, if \ $V=I_{_n}$,  we have \ $ Fix(-\eta) = \mathfrak{u}_n $).

Taking into account Remark-Definition \ref{intersec-SVD}, we obtain that

$\langle  V  \rangle_{_{\mathfrak{g}}} = \langle  V \rangle_{_{\mathfrak{g\!l}_n\!(\mathbb{C})}} \cap \ \mathfrak{g}$\ \ \ \ \ \ \ \ \ \ \ and \ \ \ \ \ \ \ \ \ 
$\preccurlyeq V \succcurlyeq_{_{\mathfrak{g}}} \ = \ \preccurlyeq  V \succcurlyeq_{_{\mathfrak{g\!l}_n\!(\mathbb{C})}} \cap \ \mathfrak{g}$

are SVD-closed real Lie subalgebras of $\mathfrak{g}$,
for every $V \in U_n$, and for every SVD-closed real Lie subalgebra $\mathfrak{g}$ of $\mathfrak{gl}_n (\mathbb{C})$.
\ \ In particular, \ for \ $\mathfrak{g} = \mathfrak{u_n}$, \ we deduce that

$Fix\big(Ad_{_V}\circ (-\eta)\big) \cap \mathfrak{u}_n =Fix\big(Ad_{_V}\big) \cap \mathfrak{u}_n =  \langle  V  \rangle_{_{\mathfrak{u}_n}}$ \ \ and

$ Fix\big(Ad_{_V}\circ (-\tau)\big)  \cap \mathfrak{u}_n  = Fix\big(Ad_{_V}\circ \mu \big)  \cap \mathfrak{u}_n = \preccurlyeq   V \succcurlyeq_{_{\mathfrak{u}_n}} $ 

are SVD-closed Lie subalgebras of $\mathfrak{u}_n$, for every $V \in U_n$.

Other particular SVD-closed real Lie subalgebras of $\mathfrak{gl}_n (\mathbb{C})$ are the following:

 $\mathfrak\mathfrak{gl}_n (\mathbb{R}) = Fix(\mu)$; \ \ \ \ \ \ \ \ \ \ $\mathfrak{so}_n (\mathbb{C}) = Fix(-\tau)$;
\ \ \ \ \ \ \ \ \ \ $\mathfrak{so}_n  = \mathfrak{u}_n \cap \mathfrak\mathfrak{gl}_n (\mathbb{R})$;

$\mathfrak{sp}_{2n} (\mathbb{C}) = Fix\big(Ad_{_{\Omega_{_n}}}\!\circ (-\tau)\big)$; \ \ \ \ \
$\mathfrak{sp}_n =  \mathfrak{sp}_{2n} (\mathbb{C}) \cap \mathfrak{u}_{2n}$;
\ \ \ \  \ $\mathfrak{su}_2 = \mathfrak{sp}_2 (\mathbb{C}) \cap \mathfrak{u}_2$;

$\mathfrak{sp}_{2n} (\mathbb{R}) = \mathfrak{sp}_{2n} (\mathbb{C})  \cap \mathfrak\mathfrak{gl}_n (\mathbb{R})$; \ \ \ \ \ \ \ \ \ \ \ \ \ \ \ \ \ 
$\mathfrak{u}_{(p, q)} = Fix\big(Ad_{_{J^{(p,q)}}}\! \circ (-\eta)\big)$;

$\mathfrak{so}_{(p,q)} (\mathbb{C}) = Fix\big(Ad_{_{J^{(p,q)}}} \circ (-\tau)\big)$; \ \ \ \ \ \ \ \ \ \ \ \ \ \ 
$\mathfrak{so}_{(p,q)}= \mathfrak{so}_{(p,q)} (\mathbb{C}) \cap \mathfrak{gl}_{(p+q)} (\mathbb{R})$.
\end{examples}

\begin{rem}\label{controesempi}
If $n \ge 3$, the following are not SVD-closed real Lie subalgebras of $\mathfrak{gl}_n (\mathbb{C})$: 

$\mathfrak{su}_n$,  \ \ \ \ \ \ $\mathfrak{sl}_n (\mathbb{C}) = \{M \in \mathfrak{gl}_n (\mathbb{C}) : tr(M) =0\}$, \ \ \ \ \ \ $\mathfrak{sl}_n (\mathbb{R}) := \mathfrak{sl}_n (\mathbb{C}) \cap \mathfrak\mathfrak{gl}_n (\mathbb{R})$.

We check it only  for $\mathfrak{su}_3$; the generalization to $n > 3$ and the other cases go similarly.

The SVD-components of the matrix $D=\begin{pmatrix}
{\bf i} & 0 & 0 \\ 
0 &  {\bf i} & 0\\
0 & 0 & -2 {\bf i}
\end{pmatrix}$ are
 $\begin{pmatrix}
0 & 0 & 0 \\ 
0 &  0 & 0\\
0 & 0 & - {\bf i}
\end{pmatrix}$
and 
$\begin{pmatrix}
{\bf i} & 0 & 0 \\ 
0 &  {\bf i} & 0\\
0 & 0 & 0
\end{pmatrix}$
(being $1$ and $2$ the singular values of $D$); since $D \in \mathfrak{su}_3$, while its SVD-components do not belong to $\mathfrak{su}_3$, we can conclude that the Lie algebra $\mathfrak{su}_3$ is not SVD-closed.
\end{rem}

\begin{prop}\label{CartanSVD}
Let $\mathfrak{g}$ be an SVD-closed real Lie subalgebra of $\mathfrak{gl}_n (\mathbb{C})$.

a) For every $W \in \mathfrak{u}_n$, we have that  $\langle  W  \rangle_{_{\mathfrak{g}}}$ is an SVD-closed Lie subalgebra of $\mathfrak{g}$.

b) If $\mathfrak{g}$ is the Lie algebra of a closed subgroup of $U_n$, then every Cartan subalgebra of $\mathfrak{g}$ is SVD-closed.
\end{prop}

\begin{proof}
Clearly, if \ $YW=WY$ then $Ye^{sW} = e^{sW}Y$, for every $s \in \mathbb{R}$; conversely, if $Ye^{sW} = e^{sW}Y$ for every $s \in \mathbb{R}$, then,  differentiating with respect to $s$ and putting $s=0$, we get $YW=WY$. Hence $\langle W \rangle_{_{\mathfrak{g}}} = \mathfrak{g}\cap \big[ \bigcap\limits_{s \in \mathbb{R}} Fix \big (Ad_{_{\exp(sW)}} \big) \big]$. We get (a), since $\exp(sW) \in U_n$, for every $s \in \mathbb{R}$.
Part (b) follows from part (a), via \cite[Lemma\,5.7 p.\,100]{Sepa2007}.
\end{proof}

\section{SVD-closed subgroups of $U_n$}\label{SVD-gruppi}

\begin{remdef}\label{svdgroups}
We say that any subgroup of $GL_n (\mathbb{C})$ is \emph{SVD-closed} if it is closed in $GL_n (\mathbb{C})$ and its Lie algebra is an SVD-closed real Lie subalgebra of $\mathfrak{gl}_n (\mathbb{C})$. Note that, by Examples \ref{esempi} and Remarks \ref{parentesi-angolari2} (b), the subgroups of $U_n$, defined by 

$\preccurlyeq  V  \succcurlyeq_{_{U_n}} = \{X \in U_n : XV=V \overline{X} \}=\{X \in U_n : XVX^T=V \}$ and

$\langle  V  \rangle_{_{U_n }} = \{X \in U_n : XV=VX \}$, \ are SVD-closed, \ for every matrix $V \in U_n$.

By Remark-Definition \ref{intersec-SVD}, the intersection of SVD-closed subgroups of $GL_n (\mathbb{C})$ is an SVD-closed subgroup of $GL_n (\mathbb{C})$; indeed, it is known that its Lie algebra is the intersection of Lie algebras of all SVD-closed subgroups (\cite[Cor.\,3 p.\,307]{Bou1975}).
In the Sections \ref{Sect-U-n} and \ref{Sect-Q-in-O-n}, we will study the sets of generalized principal logarithms of matrices of the groups 

$\langle  V  \rangle_{_{U_n }}$, \ where $V \in U_n$,
\  and \ 
$\preccurlyeq  Q  \succcurlyeq_{_{SU_n}}\  = \  \preccurlyeq  Q  \succcurlyeq_{_{U_n}} \cap \ SU_n$, \ where $Q \in O_n$. 

Note that we can obtain some classical Lie groups as follows: 

$U_n= \langle I_{_n} \rangle_{_{U_n}}, \ \ \ \ SO_n=\ \preccurlyeq  I_{_n}\succcurlyeq_{_{SU_n}}, \ \ \ \ Sp_n= \ \preccurlyeq \Omega_{_n} \succcurlyeq_{_{SU_{2n}}}$,

$U_{(p,n-p)}\cap\ U_n= \langle J^{(p,n-p)} \rangle_{_{U_n}}, \ \ \ SO_{(p, n-p)} (\mathbb{C}) \cap U_n =\  \preccurlyeq J^{(p,n-p)} \succcurlyeq_{_{SU_n}},$ 

for $p=0, \cdots, n.$
\ We need some preliminary results.
\end{remdef}

\begin{prop}\label{UnQ}
Let $V \in U_n$; denote by $\lambda_{_1}$ (with multiplicity $n_{_1}$), $\cdots, \lambda_{_r}$ (with multiplicity $n_{_r}$) its  distinct eigenvalues, and choose $R \in U_n$ such that $V = Ad_{_R}\big(\bigoplus\limits_{j=1}^r \lambda_{_j} I_{_{n_j}}\big)$. Then
$\langle  V  \rangle_{_{U_n}} = Ad_{_R} \big( \bigoplus_{j=1}^r U_{n_{_j}} \big)$ and it is a (compact) connected SVD-closed   subgroup of $U_n$, whose Lie algebra is  $\langle  V \rangle_{_{\mathfrak{u}_n}}= Ad_{_R}\big( \bigoplus_{j=1}^r \mathfrak{u}_{n_{_j}} \big)$.
\end{prop}

\begin{proof}
The equality $\langle  V  \rangle_{_{U_n}} = Ad_{_R} \big( \bigoplus_{j=1}^r U_{n_{_j}} \big)$ easily follows from Lemma \ref{anticomm} (b). This implies that $\langle  V \rangle_{_{U_n}}$ is compact and connected. As noted in Remark-Definition \ref{svdgroups}, $\langle  V  \rangle_{_{U_n}}$ is SVD-closed too. Clearly, its Lie algebra is $\langle  V  \rangle_{_{\mathfrak{u}_n}} = Ad_{_R}\big( \bigoplus_{j=1}^r \mathfrak{u}_{n_{_j}} \big)$. 
\end{proof}

\begin{lemma}\label{SUn(V)}
Let $V$ any matrix of $U_n$. Then
$\preccurlyeq  V  \succcurlyeq_{_{SU_n}}$ is an SVD-closed subgroup of $U_n$, whose Lie algebra is \ $\preccurlyeq V \succcurlyeq_{_{\mathfrak{u}_n}}\ =\ \preccurlyeq V \succcurlyeq_{_{\mathfrak{s\!u}_n}}$.
\end{lemma}

\begin{proof}
The Lie algebra of $\preccurlyeq  V  \succcurlyeq_{_{SU_n}}$ is \ $\preccurlyeq V \succcurlyeq_{_{\mathfrak{s\!u}_n}}\ \subseteq\ \preccurlyeq V \succcurlyeq_{_{\mathfrak{u}_n}}$ and this last is SVD-closed, so it suffices to prove the reverse inclusion. 
If $X \in \ \preccurlyeq V \succcurlyeq_{_{\mathfrak{u}_n}}$, \ being $V^* X V = \overline{X}$, \ then $X$ is similar to its complex conjugate $\overline{X}$ and so, by \cite[Cor.\,3.4.1.7 p.\,202]{HoJ2013}, $X$ is similar to a real matrix; therefore $X$ has real trace; since any skew-hermitian matrix has trace with zero real part, we conclude that the trace of $X$  is zero, i.e. $X \in \ \preccurlyeq V \succcurlyeq_{_{\mathfrak{s\!u}_n}}.$ 
\end{proof}

In the next results, we will need the matrices \ $W_{_{(p, q)}}$, \ $E_{_\varphi}^{(p, q)}$ \ and \ $J^{(p, q)}$ defined in Notations \ref{notazioni}(a).

\begin{lemma}\label{rem-Jp}
If $p = 0, 1, \cdots, n$ , we have $O_{(p, n-p)} (\mathbb{C}) \cap U_n = Ad_{_{W_{_{(p, n-p)}}}}(O_n)$ \ and \ $SO_{(p, n-p)} (\mathbb{C}) \cap U_n = Ad_{_{W_{_{(p, n-p)}}}}(SO_n)$.
\end{lemma}

\begin{proof}
Let $W:=W_{_{(p,n-p)}}$. Then the statements follow from Remarks \ref{parentesi-angolari2} (a), since 

$\preccurlyeq I_{_n} \succcurlyeq_{_{U_n}} = O_n$ , \ \ \ \ $\preccurlyeq I_{_n} \succcurlyeq_{_{SU_n}} = SO_n$ ,  \ \ \ \ $\preccurlyeq J^{(p, n-p)} \succcurlyeq_{_{U_n}} = O_{(p, n-p)} (\mathbb{C}) \cap U_n$,

$\preccurlyeq J^{(p, n-p)} \succcurlyeq_{_{SU_n}} = SO_{(p, n-p)} (\mathbb{C}) \cap U_n$, \ \ $W I_{_n} W^T = J^{(p, n-p)}
$ and the groups $U_n, SU_n$ are $Ad_{_W}$-invariant.
\end{proof}

\begin{lemma}\label{rem-Theta}
For every \ $\varphi \in \mathbb{R}$ \ and \ $ p = 0, 1, \cdots, n$, \ we have 

\ \ $\preccurlyeq E_{_{\varphi}}^{(p, n-p)}  \succcurlyeq_{_{U_{2n}}} = Ad_{_{W_{_{(2p,2n-2p)}}}}( \ \preccurlyeq E_{_{\varphi}}^{\oplus n}  \succcurlyeq_{_{U_{2n}}})$\ \ \ \ and

\ \ $\preccurlyeq E_{_{\varphi}}^{(p, n-p)}  \succcurlyeq_{_{SU_{2n}}} = Ad_{_{W_{_{(2p,2n-2p)}}}}( \ \preccurlyeq E_{_{\varphi}}^{\oplus n}  \succcurlyeq_{_{SU_{2n}}})$.
\end{lemma}

\begin{proof}
Let $W:=W_{_{(2p,2n-2p)}}$. \ \ The groups $U_{2n}$ and $SU_{2n}$ are $Ad_{_W}$-invariant and 

$W \ E_{_{\varphi}}^{\oplus n} \ W^T = E_{_{\varphi}}^{(p, n-p)}$\ ; hence, by Remarks \ref{parentesi-angolari2} (a), we get the statements .
\end{proof}

\begin{lemma}\label{lemma-Theta-plog}
Fix $\varphi \in [0, 2\pi)$, with $\varphi \ne \dfrac{\pi}{2}$ and  $\varphi \ne \dfrac{3}{2} \pi$; consider the matrix $ E_{_{\varphi}}^{\oplus n}$. Then a matrix $A \in \mathfrak{gl}_{2n} (\mathbb{C})$ anticommutes with $ E_{_{\varphi}}^{\oplus n}$ if and only if $A={\bf 0}_{_{2n}}$.
\end{lemma}

\begin{proof}
Assume first $n=1$, so $ E_{_{\varphi}}^{\oplus n} = E_{_{\varphi}} = 
\begin{pmatrix}
\cos(\varphi) & -\sin(\varphi) \\ 
\sin(\varphi) & \cos(\varphi)
\end{pmatrix}$. If a matrix 

$A=
\begin{pmatrix}
\alpha & \beta \\ 
\gamma & \delta
\end{pmatrix} 
 \in \mathfrak{gl}_2 (\mathbb{C})$ anticommutes with  $E_{_{\varphi}}$, then
$
\begin{cases}
2 \alpha \cos(\varphi) =  (\gamma -\beta) \sin(\varphi)\\
2 \delta \cos(\varphi) =  (\gamma -\beta) \sin(\varphi)\\
2 \gamma \cos(\varphi) = -(\alpha + \delta) \sin(\varphi)\\
2 \beta \cos(\varphi) = (\alpha + \delta) \sin(\varphi)
\end{cases}
$.

Since $\cos(\varphi) \ne 0$, the previous conditions give: $\alpha= \delta$ and $\beta = -\gamma$, i.e. $A= \alpha I_{_2} + \gamma \Omega$. But this last matrix also commutes with the nonsingular matrix $E_{_{\varphi}}$ and so, $A$ must be the null matrix.

If $n \ge 2$, we write any matrix of $A \in \mathfrak{gl}_{2n} (\mathbb{C})$  as $A:= (A_{_{ij}})$, with
$n^2$ square blocks $A_{_{ij}}$ of order $2$. A direct computation shows that, if $A$ anticommutes with $ E_{_{\varphi}}^{\oplus n}$, then each block $A_{_{ij}}$ anticommutes with $E_{_{\varphi}}$; hence, the proof follows from the case $n=1$.
\end{proof}

\begin{lemma}\label{SUphi}
Fix $\varphi \in (0, 2 \pi)$ with $\varphi \ne \dfrac{\pi}{2}$, $\varphi \ne \pi$ and  $\varphi \ne \dfrac{3}{2} \pi$.
Then we have 
 
$\preccurlyeq E_{_{\varphi}}^{(p, n-p)}  \succcurlyeq_{_{SU_{2n}}} \ = \ \ \preccurlyeq E_{_{\varphi}}^{(p, n-p)}  \succcurlyeq_{_{U_{2n}}} \ =\ \ Ad_{_{W_{_{(2p,2n-2p)}}}}(U_n)$, 
\ \ for every \ $p= 0, \cdots , n$, 
in which we put (consistently with Remarks \ref{identificazioni} (a)) \ \ $U_n=\mathfrak{gl}_n(\mathbb{C})\cap SO_{2n}\subset SU_{2n}$.
\end{lemma}

\begin{proof}
By Lemma \ref{rem-Theta}, we have to prove that $\preccurlyeq E_{_{\varphi}}^{\oplus n}  \succcurlyeq_{_{SU_{2n}}}= \ \preccurlyeq E_{_{\varphi}}^{\oplus n}  \succcurlyeq_{_{U_{2n}}} = U_n$. 
For, \ a complex matrix $X = X_{_1} + {\bf i} X_{_2}$ ($X_{_1}, X_{_2}$ real matrices) satisfies the condition $X \, E_{_{\varphi}}^{\oplus n} =  \, E_{_{\varphi}}^{\oplus n} \overline{X}$ if and only if $X_{_1} \, E_{_{\varphi}}^{\oplus n}=  \, E_{_{\varphi}}^{\oplus n} X_{_1}$ \ \ and \ \ $X_{_2} \, E_{_{\varphi}}^{\oplus n}= -  \, E_{_{\varphi}}^{\oplus n} X_{_2}$ and, by Lemmas \ref{lemma-Theta-plog} and \ref{commutare-phi}, this is equivalent to say that $X \in \mathfrak{gl}_n (\mathbb{C}) \subseteq \mathfrak{gl}_{2n} (\mathbb{R})$ (and, in this case, $det(X)\geq 0$). Hence, by Remarks \ref{identificazioni} (a), we get $\preccurlyeq E_{_{\varphi}}^{\oplus n}  \succcurlyeq_{_{SU_{2n}}} = \mathfrak{gl}_n (\mathbb{C}) \cap SU_{2n} =  \mathfrak{gl}_n (\mathbb{C}) \cap SO_{2n} = U_n$ and similarly,
\ $\preccurlyeq E_{_{\varphi}}^{\oplus n}  \succcurlyeq_{_{U_{2n}}} = \mathfrak{gl}_n (\mathbb{C}) \cap U_{2n} = \mathfrak{gl}_n (\mathbb{C}) \cap SO_{2n} = U_n$.
\end{proof}

\begin{lemma}\label{lemma-UH}
Remembering Remarks \ref{identificazioni} (b), \ we have \ \ \ 

$\preccurlyeq \Omega^{\oplus n} \succcurlyeq_{_{SU_{2n} \!}}\ = \ \preccurlyeq \Omega^{\oplus n} \succcurlyeq_{_{U_{2n} \!}}\  =\ U_n (\mathbb{H})$ \ \ \ \ and
\ \ \ \ $\preccurlyeq \Omega^{\oplus n} \succcurlyeq_{_{\mathfrak{s\!u}_{2n} \!}}\ = \ \preccurlyeq \Omega^{\oplus n} \succcurlyeq_{_{\mathfrak{u}_{2n} \!}}\  =\ \mathfrak{u}_n (\mathbb{H})$.
\end{lemma}

\begin{proof}
Any matrix $X = Y + {\bf i} Z \in \mathfrak{gl}_{2n}(\mathbb{C})$ (with $Y, Z \in \mathfrak{gl}_{2n}(\mathbb{R})$) satisfies the condition $X \, \Omega^{\oplus n} =  \, \Omega^{\oplus n} \overline{X}$ if and only if $Y \Omega^{\oplus n} = \Omega^{\oplus n} Y$ and $Z\Omega^{\oplus n}= -\Omega^{\oplus n}Z$. A direct computation shows that these conditions on $Y$ and $Z$ are equivalent to say that  $Y=(Y_{_{ij}})$ and 
$Z = (Z_{_{ij}})$ are block matrices, whose blocks $Y_{_{ij}} , Z_{_{ij}}$ are $2\times2$ real matrices of the form:
$
Y_{_{ij}}=\begin{pmatrix}
a_{_{ij}} & -b_{_{ij}} \\ 
b_{_{ij}} & a_{_{ij}}
\end{pmatrix} 
,$ \ 
$
Z_{_{ij}}=\begin{pmatrix}
c_{_{ij}} & d_{_{ij}} \\ 
d_{_{ij}} & -c_{_{ij}}
\end{pmatrix} 
$, for $i,j = 1, \cdots , n.$ These last conditions are equivalent to say that 
$X = (X_{_{ij}})$ is a block matrix, with $n^2$ blocks of the form:
$X_{_{ij}}=\begin{pmatrix}
z_{_{ij}} & -w_{_{ij}} \\ 
\overline{w}_{_{ij}} & \overline{z}_{_{ij}}
\end{pmatrix},$  and, by Remarks \ref{identificazioni} (b), this is equivalent to say that $X\in \mathfrak{gl}_n (\mathbb{H})$.
Hence \ $\preccurlyeq \Omega^{\oplus n} \succcurlyeq_{_{SU_{2n} \!}}\ = \ SU_{2n} \cap \mathfrak{gl}_n (\mathbb{H}) = \ U_n (\mathbb{H})\ = \ U_{2n} \cap \mathfrak{gl}_n (\mathbb{H}) = \ \preccurlyeq \Omega^{\oplus n} \succcurlyeq_{_{U_{2n} \!}}\ $ and, by Remarks \ref{parentesi-angolari2} (b), we also get \  $\preccurlyeq \Omega^{\oplus n} \succcurlyeq_{_{\mathfrak{s\!u}_{2n} \!}}\ = \ \preccurlyeq \Omega^{\oplus n} \succcurlyeq_{_{\mathfrak{u}_{2n} \!}}\  =\ \mathfrak{u}_n (\mathbb{H})$.
\end{proof}

\begin{rems}\label{oss-prima-teor-comp}
a) For any  $Q \in O_n$, there exists a matrix $A \in O_n$ such that $Q = Ad_{_A}(\mathcal{J}) = A \mathcal{J}A^T $, where $\mathcal{J}$ is a matrix of the form $\mathcal{J}:=J^{(p, q)} \oplus \big( \bigoplus\limits_{j=1}^h  E_{_{\varphi_{_j}}}^{(\mu_{_j}, \, \nu_{_j})} \big) \oplus \Omega^{\oplus k}$, 

with  $0 < \varphi_{_1} < \varphi_{_2} < \cdots < \varphi_{_h} < \dfrac{\pi}{2}$; \ \ $p+q + 2 \sum\limits_{j=1}^h (\mu_{_j}+\nu_{_j}) + 2k =n$; \  \  $p, q,  k, \mu_{_j}, \nu_{_j} \ge 0$; \ \  $\mu_{_j} + \nu_{_j} \ge 1$  \ \ (see for instance \cite[Rem.-Def.\,1.8]{DoPe2021}, where we called $\mathcal{J}$ the \emph{real Jordan auxiliary  form} of $Q$). 
Hence the (possible) eigenvalues of $Q$ and their multiplicities are the following: $1$ of multiplicity $p \ge 0$; $-1$ of multiplicity $q \ge 0$;  $\pm {\bf i}$  both of multiplicity $k \ge 0$;
when $h > 0$, $e^{\pm {\bf i} \varphi_{_j}}$ both of multiplicity $\mu_{_j} \ge 0$ and  $e^{\pm {\bf i} (\pi -\varphi_{_j})} = - e^{\mp {\bf i} \varphi_{_j}}$ both of
multiplicity  $\nu_{_j} \ge 0$, for every $j= 1, \cdots , h$. The condition $\mu_{_j} + \nu_{_j} \ge 1$ is equivalent to say that $e^{\pm {\bf i} \varphi_{_j}}$ or $e^{\pm {\bf i} (\pi -\varphi_{_j})}$ (and possibly both) are  effective eigenvalues of $Q$.

\smallskip

b) If $Q, A,  \mathcal{J} \in O_n$ are as in (a), we have
$Ad_{_A} \big(I_{_1} \oplus (-I_{_{(n-1)}}\big)) \in \ \preccurlyeq Q \succcurlyeq_{_{SU_n}}$ if and only if $n$ is odd.
\ Indeed, if $n$ is odd, the real matrix $Q$ has at least one real eigenvalue.
\end{rems}

\begin{prop}\label{propprinc}
Let $Q \in O_n$; denote its eigenvalues (with their multiplicities) and the matrices $A , \mathcal{J} \in O_n$  as in Remarks \ref{oss-prima-teor-comp} (a). If $Z$ is the $n\times n$ unitary matrix defined by 

$Z :=\ A \bigg(W_{(p, q)}\oplus \bigg[\bigoplus\limits_{j=1}^h W_{(2\mu_{_j}, 2\nu_{_j})}\bigg] \oplus I_{_{2k}}\bigg)$,
then 

$\preccurlyeq Q \succcurlyeq_{_{U_n}}\ = \ Ad_{_Z}\bigg(O_{(p+q)}\oplus \bigg[\bigoplus\limits_{j=1}^h U_{(\mu_{_j}+ \nu_{_j})} \bigg]\oplus U_k (\mathbb{H})\bigg)$\ ,

$\preccurlyeq Q \succcurlyeq_{_{SU_n}}\ = \ Ad_{_Z}\bigg(SO_{(p+q)}\oplus \bigg[\bigoplus\limits_{j=1}^h U_{(\mu_{_j}+ \nu_{_j})} \bigg]\oplus U_k (\mathbb{H})\bigg)$\ ,

and they are (compact) SVD-closed subgroups of $U_n$, whose common Lie algebra is  

$\preccurlyeq Q \succcurlyeq_{_{\mathfrak{s\!u}_n}} \ = \ \preccurlyeq Q \succcurlyeq_{_{\mathfrak{u}_n}} \ = \ Ad_{_Z}\bigg(\mathfrak{so}_{(p+q)}\oplus \bigg[\bigoplus\limits_{j=1}^h \mathfrak{u}_{(\mu_{_j}+ \nu_{_j})} \bigg]\oplus \mathfrak{u}_k (\mathbb{H})\bigg)$.

The group $\preccurlyeq Q \succcurlyeq_{_{U_n}}$ is connected if $Q$ has no real eigenvalues, otherwise it has two connected components. In any case, $\preccurlyeq Q \succcurlyeq_{_{SU_n}}$ is the connected component of $\preccurlyeq Q \succcurlyeq_{_{U_n}}$ containing the identity $I_{_n}$.
\end{prop}

\begin{proof}
From Remark-Definition \ref{svdgroups} and Lemma \ref{SUn(V)}, it follows that the groups $\preccurlyeq Q \succcurlyeq_{_{U_n}}$ and $\preccurlyeq Q \succcurlyeq_{_{SU_n}}$ are SVD-closed and their  common Lie algebras is  $\preccurlyeq Q \succcurlyeq_{_{\mathfrak{u}_n}}=\ \preccurlyeq Q \succcurlyeq_{_{\mathfrak{s\!u}_n}}.$ 
By Remarks \ref{parentesi-angolari2} (a), we have
$\preccurlyeq Q \succcurlyeq_{_{U_n}}=Ad_{_A}(\preccurlyeq \mathcal{J} \succcurlyeq_{_{U_n}}),$\ \  
$\preccurlyeq Q \succcurlyeq_{_{SU_n}}=Ad_{A}(\preccurlyeq \mathcal{J} \succcurlyeq_{_{SU_n}}).$ Now we determine the groups $\preccurlyeq \mathcal{J} \succcurlyeq_{_{U_n}}$ and $\preccurlyeq \mathcal{J} \succcurlyeq_{_{SU_n}}$.
A matrix $X = X_{_1} + {\bf i} X_{_2}\in \mathfrak{gl}_n(\mathbb{C})$ (with $X_{_1}, X_{_2} \in \mathfrak{gl}_n(\mathbb{R})$) satisfies the condition $X \mathcal{J} = \mathcal{J} \overline{X}$ if and only if $X_{_1}\mathcal{J}= \mathcal{J} X_{_1}$ \ \ and \ \ $X_{_2}\mathcal{J}= - \mathcal{J} X_{_2}$. 
\ By Lemma \ref{anticomm} (b), the condition $X_{_1}\mathcal{J}= \mathcal{J} X_{_1}$ implies that

$X_{_1} = Y_{_0} \oplus \bigg[ \bigoplus_{j=1}^h Y_{_j} \bigg] \oplus Y_{_{(h+1)}}$, \ \ \  where $Y_{_0} \in \mathfrak{gl}_{(p+q)} (\mathbb{R})$, \ \  \ $Y_{_j} \in \mathfrak{gl}_{(2\mu_{_j}+2\nu_{_j})} (\mathbb{R})$ \  for every $j= 1, \cdots , h$ \ and \ \  $Y_{_{(h+1)}} \in \mathfrak{gl}_{2k} (\mathbb{R})$.
By Lemma \ref{anticomm} (a), the condition $X_{_2}\mathcal{J}= - \mathcal{J} X_{_2}$ implies that also the matrix $X_{_2}$ must be block-diagonal, with blocks of the same type as the blocks of $X_{_1}$. Therefore, if $X$ satisfies the condition $X \mathcal{J} = \mathcal{J} \overline{X}$, then $X$ is block-diagonal with similar blocks, this time complex instead of real. Of course, $X$ is unitary if and only if each single block is unitary too. 
Then, setting $ U= W_{(p, q)}\oplus \bigg[\bigoplus\limits_{j=1}^h W_{(2\mu_{_j}, 2\nu_{_j})}\bigg] \oplus I_{_{2k}}$ and taking into account also Lemmas \ref{rem-Jp}, \ref{SUphi} and \ref{lemma-UH} , we obtain

$
\preccurlyeq \mathcal{J} \succcurlyeq_{_{U_n}} \ = \ \preccurlyeq J^{(p, q)} \succcurlyeq_{_{U_{(p+q)}}} \oplus \bigg[\bigoplus\limits_{j=1}^h \preccurlyeq E_{_{\varphi_{_j}}}^{(\mu_j, \nu_j)} \succcurlyeq_{_{U_{(2\mu_{_j}+2\nu_{_j})}}}
  \bigg] \oplus \preccurlyeq \Omega^{\oplus k} \succcurlyeq_{_{U_{2k}}}
\ =$

$\preccurlyeq J^{(p, q)} \succcurlyeq_{_{U_{(p+q)}}} \oplus \bigg[\bigoplus\limits_{j=1}^h \preccurlyeq E_{_{\varphi_{_j}}}^{(\mu_j, \nu_j)} \succcurlyeq_{_{SU_{(2\mu_{_j}+2\nu_{_j})}}}
  \bigg] \oplus \preccurlyeq \Omega^{\oplus k} \succcurlyeq_{_{SU_{2k}}}=$
  
$Ad_{_U}\bigg(O_{(p+q)}\oplus \bigg[\bigoplus\limits_{j=1}^h U_{(\mu_{_j}+ \nu_{_j})} \bigg]\oplus U_k (\mathbb{H})\bigg)$; 
  
$\preccurlyeq \mathcal{J} \succcurlyeq_{_{SU_n}}  \ = \ \preccurlyeq J^{(p, q)} \succcurlyeq_{_{SU_{(p+q)}}} \oplus \bigg[\bigoplus\limits_{j=1}^h \preccurlyeq E_{_{\varphi_{_j}}}^{(\mu_j, \nu_j)} \succcurlyeq_{_{SU_{(2\mu_{_j}+ 2\nu_{_j})}}}
  \bigg] \oplus \preccurlyeq \Omega^{\oplus k} \succcurlyeq_{_{SU_{2k}}}=$
  
$ Ad_{_U}\bigg(SO_{(p+q)}\oplus \bigg[\bigoplus\limits_{j=1}^h U_{(\mu_{_j}+ \nu_{_j})} \bigg]\oplus U_k (\mathbb{H})\bigg)$. 

\smallskip

From these equalities, easily follow the statements that still remain to be proved.
\end{proof}

\section{Generalized principal $\mathfrak{g}$-logarithms}\label{Sect-plog}

\begin{defi}\label{Def-SVD}
Let $G$ be a connected closed subgroup of $GL_n (\mathbb{C})$, whose Lie algebra is $\mathfrak{g} \subseteq \mathfrak{gl}_n (\mathbb{C})$ . If $M \in G$, we say that a matrix $L \in \mathfrak{g}$ is a \emph{generalized principal} $\mathfrak{g}$-\emph{logarithm} of $M$, if \ $\exp(L) = M$ \ and \ $- \pi \le Im(\lambda) \le\pi$, \ for every eigenvalue $\lambda$ of $L$.

We denote by $\mathfrak{g}$--$plog(M)$ the set of all generalized principal $\mathfrak{g}$-logarithms of any $M \in G$.
\end{defi}

\begin{rems}\label{confronta-Higham}
a) In Introduction, we compared the previous definition with the usual definition of \emph{principal logarithm} of a matrix $M \in GL_n (\mathbb{C})$ without negative eigenvalues, in which case the set $\mathfrak{gl}_n (\mathbb{C})$--$plog(M)$ consists of a unique matrix (\cite[Thm.\,1.31]{Hi2008}).

b) If $G$ is any connected closed subgroup of $GL_n (\mathbb{C})$, with Lie algebra $\mathfrak{g}  \subseteq \mathfrak{gl}_n (\mathbb{C})$, then $\rho(G)$ is a connected closed subgroup of $GL_{2n} (\mathbb{ R}) \subset GL_{2n} (\mathbb{C})$, having $\rho(\mathfrak{g}) \subset \mathfrak{gl}_{2n} (\mathbb{R}) \subset\mathfrak{gl}_{2n} (\mathbb{C})$ as Lie algebra, where $\rho$ is the decomplexification map. Remembering the relationship between the eigenvalues of $Z$ and $\rho(Z)$ (see Remarks \ref{identificazioni} (a)), we easily get that 
\\ $\rho\big(\mathfrak{g}$--$plog(M)\big) = \rho(\mathfrak{g})$--$plog(\rho(M))$, \ \ for every $M \in G$.
\end{rems}

\begin{lemma}\label{lemma-plog-coniug}
Let $G, H$ be connected closed subgroups of $GL_n (\mathbb{C})$ such that $G=Ad_{_A}(H)$, for some $A \in GL_n (\mathbb{C})$, and let 
$\mathfrak{g}, \mathfrak{h} \subseteq \mathfrak{gl}_n (\mathbb{C})$ be their Lie algebras, respectively. Then

$Ad_{_A}(\mathfrak{h}$--$plog(M)) = \mathfrak{g}$--$plog(Ad_{_A}(M))$, \  for every $M \in H$. 

In particular, if $G$ is any connected closed subgroup of $GL_n (\mathbb{C})$, we have

$Ad_{_A}(\mathfrak{g}$--$plog(M)) = \mathfrak{g}$--$plog(Ad_{_A}(M))$, \ for every $A, M \in G$. 
\end{lemma}

\begin{proof}
Note that $G=Ad_{_A}(H)$ implies that $\mathfrak{g} = Ad_{_A}(\mathfrak{h})$. Hence $B \in \mathfrak{g}$ if and only if $A^{-1} B A \in \mathfrak{h}$. Since $B$ and $A^{-1} B A$ are similar and $\exp(B) = A M A^{-1}$ if and only if $\exp(A^{-1} B A) = M$, we get: $B \in \mathfrak{g}$--$plog(Ad_{_A}(M))$ if and only if $A^{-1} B A \in \mathfrak{h}$--$plog(M)$.
\end{proof}

\begin{rem}\label{val-sing-immag}
The eigenvalues of any skew-hermitian matrix $A$ are purely imaginary; so, the generalized principal $\mathfrak{u}_n$-logarithms of any $M \in U_n$ are the skew-hermitian logarithms of $M$, whose eigenvalues all have modulus in $[0, \pi]$.
Note that, since all the eigenvalues of any $M  \in U_n$ have modulus $1$, the only possible negative eigenvalue of such $M$ is $-1$.

In this Section, given any unitary matrix $M$ of order $n$, we will denote its eigenvalues by $e^{{\bf i} \theta_{_1}}$ with multiplicity $m_{_1}$, $e^{{\bf i} \theta_{_2}}$ with multiplicity $m_{_2}$, up to $e^{{\bf i} \theta_{_p}}$ with multiplicity $m_{_p}$, where 
$\pi \ge \theta_{_1} > \theta_{_2} > \cdots > \theta_{_p} > -\pi$ \ and \ $n=\sum\limits_{j=1}^p m_{_j}.$ 
\ \ If $-1$ is not an eigenvalue of $M$ (i.e. if $\theta_{_1} < \pi$), then the eigenvalues of the unique generalized principal 
$\mathfrak{gl}_n(\mathbb{C)}$-logarithm of $M$ are exactly: ${\bf i} \theta_{_1}$ with multiplicity $m_{_1}$, ${\bf i} \theta_{_2}$ with multiplicity $m_{_2}$, up to ${\bf i} \theta_{_p}$ with multiplicity $m_{_p}$.
\ \ Instead, if $-1$ is an eigenvalue of $M$ (i.e. if $\theta_{_1} = \pi$), then the eigenvalues of any  generalized principal $\mathfrak{gl}_n(\mathbb{C)}$-logarithm $Y$ of $M$ are exactly: ${\bf i} \pi$ of multiplicity $h$, \ ${-\bf i} \pi$ of multiplicity $m_{_1}\!\!-h$ (for some $ h \in \lbrace0,1, \cdots, m_{_1} \rbrace$ depending on $Y$), ${\bf i} \theta_{_2}$ with multiplicity $m_{_2}$, up to ${\bf i} \theta_{_p}$ with multiplicity $m_{_p}$.
Note that, if $Y$ is any  generalized  principal 
$\mathfrak{u}_n$-logarithm of $M$, in any case we have
\ \ $\Vert Y \Vert_{_\phi}^2= \ -tr(Y^2) \ = \ \sum\limits_{j=1}^n m_{_j}\theta_{_j}^2 \ = \ \sum\limits_{j=1}^n m_{_j} |\log(e^{{\bf i} \theta_{_j}})|^2$. 
\end{rem}

\begin{prop}\label{esist-gplog}
Let $G$ be a connected SVD-closed subgroup of $U_n$, whose Lie algebra is $\mathfrak{g} \subseteq \mathfrak{u}_n$. Then 

a) $\mathfrak{g}$--$plog(M) \ne \emptyset$, for every $M \in G$ and, furthermore, if $-1$ is not an eigenvalue of $M$, then $\mathfrak{g}$--$plog(M)$ consists of a single element;

b) If $Y \in \mathfrak{g}$--$plog(M)$, then $\Vert Y \Vert_{_\phi} \le \Vert X \Vert_{_\phi}$,
 for every $X \in \mathfrak{g}$ such that $\exp(X) = M$; moreover the equality holds if and only if $X \in \mathfrak{g}$--$plog(M)$.
\end{prop}

\begin{proof}
a) If $M= I_{_n}$, it is clear that $\mathfrak{g}$--$plog(M)= \lbrace {\bf 0}_{_n}\!\rbrace$ and the statement holds true.

Fix $M \in G \setminus{\lbrace I_{_n}\!\rbrace}$ and denote its eigenvalues as in Remark \ref{val-sing-immag}. Since $G$ is compact and connected, we can choose  a skew-hermitian matrix $X \in \mathfrak{g} \setminus \lbrace {\bf 0}_{_n}\!\rbrace$ such that $\exp(X) = M$ (see, for instance, \cite[Ch.\,IV Thm.\,2.2]{BrtD1985}). Then, the $n$ eigenvalues of $X$ are ${\bf i}(\theta_{_1} +2 k_{_{1,1}} \pi)$, ${\bf i}(\theta_{_1} +2 k_{_{1,2}} \pi)$, $\cdots$, ${\bf i}(\theta_{_1} +2 k_{_{1,m_{_1}}} \pi)$; ${\bf i}(\theta_{_2} +2 k_{_{2,1}} \pi)$, $\cdots$, ${\bf i}(\theta_{_2} +2 k_{_{2,m_{_2}}} \pi)$; $\cdots$; up to ${\bf i}(\theta_{_p} +2 k_{_{p,1}} \pi)$, $\cdots$, ${\bf i}(\theta_{_p} +2 k_{_{p,m_{_p}}} \pi)$, where  $k_{_{h, j}} \in \mathbb{Z}$, \ for every $h, j$. 
We also denote by $\sigma_{_1} > \sigma_{_2} > \dots > \sigma_{_s} >0$ the distinct non-zero singular values of $X$. Since $X \in \mathfrak{u}_n$, there exist $\psi_{_h} \in \lbrace\theta_{_1}, \cdots, \theta_{_p}\!\rbrace$ and $t_{{_h}} \in \mathbb{Z}$ such that $\sigma_{_h}= |\psi_{{_h}}+ 2t_{{_h}}\pi|$, for every $h=1, \cdots, s$. If 
$X= \sum\limits_{h=1}^s |\psi_{_h} + 2t_{_h} \pi| X_{_h}$ is the SVD-decomposition of $X$, then every SVD-component $X_{_h}$ of $X$ belongs to $\mathfrak{g}$, because $G$ is SVD-closed.
Of course, for $h=1, \cdots, s$, we have  $| \psi_{_h} + 2 t_{_h} \pi| = \pm (\psi_{_h} + 2 t_{_h} \pi)$, and so
$X= \sum\limits_{h=1}^s (\psi_{_h} + 2t_{_h} \pi) Y_{_h} = \sum\limits_{i=1}^s \psi_{_h} Y_{_h} + \sum\limits_{i=1}^s 2 \pi t_{_h} Y_{_h},$ where $Y_{_h}=\pm X_{_h}.$ Note that, by Remarks-Definitions \ref{SVD-rec} (b), $\{Y_{_h}\}_{_{1\leq h \leq s}}$ is still an SVD-system of elements of $\mathfrak{g}\ .$
Taking into account Lemma \ref{exp-svd} and the mutual commutativity of the $Y_{_h}$'s, \  we have: 
$M= \exp(X) = \exp(\sum\limits_{h=1}^s \psi_{_h} Y_{_h}) \exp(\sum\limits_{i=1}^s  2 \pi l_{_h} Y_{_h}) =
\exp(\sum\limits_{h=1}^s \psi_{_h} Y_{_h})$.
So, if we denote $Y:= \sum\limits_{h=1}^s \psi_{_h} Y_{_h}$, we have $Y \in \mathfrak{g}$ and $M=\exp(Y)$. By Lemma \ref{autov-SVD-A}, every non-zero eigenvalue of $Y$ is of the form \ $\pm \bf i \theta_{_h}$, \ for some $h = 1, \cdots , p\ ;$ \ hence $Y$ is a  generalized  principal $\mathfrak{g}$-logarithm of $M$. 
\ By Remarks \ref{confronta-Higham} (a), \ if $-1$ is not an eigenvalue of $M$, \ the set \ $\mathfrak{g}$--$plog(M)$ necessarily reduces to the single matrix $Y$.

b) Let $X \in \mathfrak{g}$ any logarithm of $M$, with eigenvalues  as in (a), and let $Y \in \mathfrak{g}$--$plog(M)$.
 
Then, $\Vert X \Vert_{_\phi}^2 = - tr(X^2) = \sum\limits_{j=1}^p \sum\limits_{r=1}^{m_{_j}} (\theta_{_j} + 2 k_{_{j,r}} \pi)^2 = \sum\limits_{j=1}^p m_{_j} \theta_{_j}^2 + 4 \pi \sum\limits_{j=1}^p \sum\limits_{r=1}^{m_{_j}} k_{_{j,r}}(\theta_{_j} + k_{_{j,r}} \pi) =$
 
$ - tr(Y^2) + 4 \pi \sum\limits_{j=1}^p \sum\limits_{r=1}^{m_{_j}} k_{_{j,r}}(\theta_{_j} + k_{_{j,r}} \pi) = \Vert Y \Vert_{_\phi}^2 + 4 \pi \sum\limits_{j=1}^p \sum\limits_{r=1}^{m_{_j}} k_{_{j,r}}(\theta_{_j} + k_{_{j,r}} \pi) $ \ (with $k_{_{j,r}}\in \mathbb{Z}$).
 
If $\theta_{_j} \in (-\pi, \pi)$, we easily get $k_{_{j,r}}(\theta_{_j} + k_{_{j,r}} \pi) \ge 0$, with equality if and only if $k_{_{j,r}}=0$. 

If $\theta_{_1}= \pi$, clearly we get $k_{_{1,r}}(\theta_{_1} + k_{_{1,r}} \pi)= \pi k_{_{1,r}} (1+ k_{_{1,r}})\geq 0$, with equality if and only if either $k_{_{1,r}} = -1$ or $k_{_{1,r}}= 0$.
Since the case $k_{_{1,r}} = -1$ gives \ $-\mathbf{i} \pi$ \ as eigenvalue of $X$, we can conclude that $\Vert X \Vert_{_\phi}^2 \ge \Vert Y \Vert_{_\phi}^2$, and the equality holds if and only if the possible eigenvalues of $X$ are only $- {\bf i} \pi $ and ${\bf i} \theta_{_j}$ ($1 \le j \le p$), i.e. if and only if $X \in G \in \mathfrak{g}$--$plog(M)$. 
\end{proof}

\begin{rem}\label{ipotesi-SVDinv}
Assume that $n \ge 3$. As noted in Remark \ref{controesempi}, \ $SU_n$ is not SVD-closed. Moreover there are matrices $M \in SU_n$ such that $\mathfrak{su}_n$--$plog(M) = \emptyset$. This is the case of $M = e^{2\pi{\bf i}/n} I_{_n}$. Indeed, $-1$ is not an eigenvalue of $M$ (since $n \ge 3$), and hence, the unique generalized  principal $\mathfrak{gl}_n(\mathbb{C)}$--logarithm  of $M$ is $L:=\dfrac{2\pi{\bf i}}{n} I_{_n}$, whose trace is $2\pi{\bf i} \ne 0$, so $L \notin \mathfrak{su}_n$. 
Hence, the SVD-closure condition in Proposition \ref{esist-gplog} cannot be removed.
\end{rem}

\begin{thm}\label{maximal-torus}
Let $G$ be a connected SVD-closed subgroup of $U_n$, whose Lie algebra is $\mathfrak{g} \subseteq \mathfrak{u}_n$; let $M \in G$ and let $T$ be a maximal torus of $G$ containing $M$, with Lie algebra $\mathfrak{t}$.

Then there are $L_{_1}, \cdots , L_{_s} \in \mathfrak{t}$--$plog(M)$ ($s \geq 1$) such that
\ $\mathfrak{g}$--$plog(M) = \bigsqcup\limits_{j=1}^s Ad\big(\langle M \rangle_{_G}\!\big)(L_{_j})$.
\ \ Furthermore, each set \ $Ad\big(\langle M \rangle_{_G}\!\big)(L_{_j})$ is a compact submanifold of $\mathfrak{g}$,  diffeomorphic to the homogeneous space $\dfrac{ \langle M \rangle_{_G}}{\langle L_j \rangle_{_G}}$.
\end{thm}

\begin{proof}
By Proposition \ref{CartanSVD} (b), $T$ is SVD-closed, being $\mathfrak{t}$ a Cartan subalgebra of $\mathfrak{g}$; so, by Proposition \ref{esist-gplog} (a), there exists a matrix $L \in \mathfrak{t}$--$plog(M)$. Furthermore, the exponential map $\exp: \mathfrak{t} \to T$ is a Lie group homomorphism (considering $\mathfrak{t}$ as an additive Lie group), so it is a covering map (see, for instance, \cite[Prop.\,1.24]{AlBet2015}) and the fiber $\exp^{-1}(M)$ is discrete.
By Proposition \ref{esist-gplog} (b), the set $\mathfrak{t}$--$plog(M)$ is the intersection between $\exp^{-1}(M)$ and the sphere  $\lbrace W \in \mathfrak{t}: \Vert W \Vert_{_\phi}=\Vert L \Vert_{_\phi}\rbrace$, therefore it is finite. We can choose a non-empty subset $\{L_{_1}, \cdots , L_{_s} \}$ of $\mathfrak{t}$--$plog(M)$ such that $L_{_h}\notin Ad\big(\langle M \rangle_{_G}\!\big) (L_{_i})$, if $h\ne i$, and such that every $L \in \mathfrak{t}$--$plog(M)$ belongs to $Ad \big(\langle M \rangle_{_G}\!\big)(L_{_j})$, for some  $j \in \lbrace1, \cdots, s\rbrace$; \ it is clear that 
$Ad\big( \langle M \rangle_{_G}\!\big) (L_{_h}) \bigcap Ad\big( \langle M \rangle_{_G}\!\big) (L_{_i}) = \emptyset$, for every $h \ne i$.

We now prove the set equality of the statement. 

If $X= Ad_{_K} (L_{_h})$, with $K \in \langle M \rangle_{_G}$, for some $h \in \lbrace1, \cdots, s\rbrace$, then clearly $X \in \mathfrak{g}$--$plog(M)$.
Conversely, let $Y \in \mathfrak{g}$--$plog(M)$. By \cite[Thm.\,5.9 p.\,101]{Sepa2007}, there exists $Q \in G$ such that $Ad_{_Q}(Y) \in \mathfrak{t}$, so that $\exp(Ad_{_Q}(Y))=Ad_{_Q}(M) \in T$. By \cite[Lemma\,2.5 p.\,166]{BrtD1985}, there exists $H $ in the \emph{normalizer} of $T$ in $G$ such that $Ad_{_H}\!\big(Ad_{_Q}(M)\big)=M$.
Since $Ad_{_H}(\mathfrak{t}) = \mathfrak{t}$, we have $Ad_{_H}\big(Ad_{_Q}(Y)\big) \in \mathfrak{t}$, with $\exp\!\big[Ad_{_H}\big(Ad_{_Q}(Y)\big)\big]=M$; so $Ad_{_H}\!\big(Ad_{_Q}(Y)\big) \in \mathfrak{t}$--$plog(M)$. Hence, there exist $j \in \{1, \cdots, s\}$ and $P \in \langle M \rangle_{_G}$ such that $Ad_{_H}\big(Ad_{_Q}(Y)\big) = Ad_{_P}(L_{_j})$, and so, \ $Y = Ad_{_K}(L_{_j})$, with $K := Q^*\!H^*\!P \in G$. Since $M=exp(Y)= exp(L_{_j})$, we get $M=Ad_{_K}(M)$, \ \ i.e. $K \in \langle M \rangle_{_G}$, and hence $Y \in Ad\big(\langle M \rangle_{_G}\!\big)(L_{_j}).$

We conclude by Remark-Definition \ref{Ad-orbit}, since $\langle M \rangle_{_G}$ is compact and $\langle {L_j} \rangle_{_G} \subseteq  \langle M \rangle_{_G}$.
\end{proof}

\section{Closed subgroups of $U_n$ endowed with the Frobenius metric}\label{Sect-Frob-metr}

\begin{remdef}
In this Section we consider an arbitrary closed subgroup $G$ of $U_n$ and we still denote by $\phi$ the Riemannian metric on $G$, obtained by restriction of the Frobenius scalar product of $\mathfrak{gl}_n (\mathbb{C})$ (remember Notations \ref{notazioni} (e)). It is easy to check that the metric $\phi$ (called the \emph{Frobenius metric} of $G$) is bi-invariant on $G$ and that we have $\phi_{_A}(X, Y) = -tr(A^* X A^* Y)$, for every $A \in G$ and for every $X, Y \in T_{_A}(G)$. We denote by $d:=d_{_{(G,\phi)}}$ the distance on $G$ induced by $\phi$ and by $\delta (G, \phi)$ the \emph{diameter} of $G$ with respect to $d$. Of course $\delta (G, \phi) < +\infty$, because $G$ is compact.
\end{remdef}

\begin{prop}\label{Levi-Civita_geod_Un} Let $G$ be a closed subgroup of $U_n$ and let $\mathfrak{g} \subseteq \mathfrak{u}_n$ be its Lie algebra. Then $(G, \phi)$ is a globally symmetric Riemannian manifold with non-negative sectional curvature, whose Levi-Civita connection agrees with the $0$-connection of Cartan-Schouten of $G$. The geodesics of $(G, \phi)$ are the curves \ 
$\gamma(t)=P\exp({t X})$,\ 
for every $X \in \mathfrak{g}$ and $P \in G$; \ furthermore $(G, \phi)$ is a totally geodesic submanifold of $(U_n, \phi)$.
\end{prop}

For a proof of Proposition \ref{Levi-Civita_geod_Un}, we refer, for instance, to \cite[\S\,2.2]{AlBet2015}.

\begin{prop}\label{distanza-punti}
Let $G$ be a connected closed subgroup of $U_n$ and let $\mathfrak{g} \subseteq \mathfrak{u}_n$ be its Lie algebra. 
Then, for every $P_{_0}, P_{_1} \in G$, the distance $d(P_{_0}, P_{_1})$ is equal to the minimum of the set
\ $\{\ \Vert X \Vert_{_\phi} \ : \ X \in  \mathfrak{g} \mbox{ and } \exp(X)= P_{_0}^* P_{_1}\ \}.$
\end{prop}

\begin{proof}
Any geodesic segment $\gamma$ joining $P_{_0}$ and $P_{_1}$ can be parametrized by $\gamma(t) = P_{_0} \exp({t X})$ ($t \in[0,1]$), with $X \in \mathfrak{g},\  \exp(X) =  P_{_0}^*P_{_1}$, and its length is $\sqrt{-tr(X^2)}= \Vert X \Vert_{_\phi}$; 
so, we conclude by the Hopf-Rinow theorem (see, for instance, \cite[p.\,31]{AlBet2015}).
\end{proof}

\begin{rem}\label{-I_n}
Let $G$ be a connected closed subgroup of $U_n$ such that $-I_{_n} \in G$. Then $\delta(G,\phi) \ge \sqrt{n} \, \pi$.
\ \ Indeed, if $\exp(X)= - I_{_n}$, with $X \in \mathfrak{g} \subseteq \mathfrak{u}_n$, the eigenvalues of $X$ are of the form $(2k_{_j} +1) \pi {\bf i}$, with $k_{_j} \in \mathbb{Z}$, so $\Vert X \Vert_{_\phi} = \sqrt{-tr(X^2)} = \sqrt{\sum\limits_{j=1}^n (2 k_{_j} +1)^2 } \cdot \pi \ge \sqrt{n} \, \pi$.
Hence, by Proposition \ref{distanza-punti}, we have \ $\delta(G,\phi) \geq d(I_{_n}, - I_{_n}) \ge \sqrt{n} \, \pi$. 
\end{rem}

\begin{thm}\label{distanza}
Let $G$ be a connected SVD-closed subgroup of $U_n$ with Lie algebra $\mathfrak{g} \subseteq \mathfrak{u}_n$. Let $P_{_0}, P_{_1} \in G$ and let \ $\mu_{_1}, \cdots , \mu_n$ \ be the $n$ eigenvalues of $P_{_0}^* P_{_1}$. \ Then

a) \ \ $
d(P_{_0}, P_{_1})= \sqrt{\sum\limits_{j=1}^n |\log(\mu_{_j})|^2}
$ ;

b) \ \ the map: $X \mapsto \gamma(t):= P_{_0} \exp(t X)\ \ (0 \leq t \leq 1)$ \ is a bijection from $\mathfrak{g}$--$plog (P_{_0}^* P_{_1})$ onto the set of minimizing geodesic segments of \ $(G, \phi)$, with endpoints $P_{_0}$ and $P_{_1}$.
\end{thm}

\begin{proof}
Part (a) follows from Propositions \ref{distanza-punti}, \ref{esist-gplog} and Remark \ref{val-sing-immag}; we also get (b), since the geodesic path: $t \mapsto P_{_0} \exp(t X)$ is minimizing if and only if $X \in \mathfrak{g}$--$plog (P_{_0}^* P_{_1}).$
\end{proof}

\begin{cor}\label{cor-diam}
Let $G$ be a connected SVD-closed subgroup of $U_n$.
Then 

a) $\delta(G,\phi) \le \sqrt{n} \, \pi$ and the equality holds if and only if $-I_{_n}\in G$; 

b) if $-I_{_n}\in G$, we have $d(P_{_0}, P_{_1}) = \delta(G,\phi)$ (with $P_{_0}, P_{_1} \in G$) if and only if $P_{_1} = -P_{_0}$.
\end{cor}

\begin{proof}
By Theorem \ref{distanza} (a), we easily get the inequality in (a), while, if $-I_{_n}\in G$, the equality follows from Remark \ref{-I_n}.
Conversely, assume that the equality holds. Since $G$ is compact, by Theorem \ref{distanza},  there exist $P_{_0}, P_{_1} \in G$ such that $ \sqrt{n} \, \pi =
d(P_{_0}, P_{_1})= \sqrt{\sum\limits_{j=1}^n |\log(\mu_{_j})|^2}
$, where $\mu_{_1}, \cdots , \mu_{_n}$ are the eigenvalues of  $P_{_0}^* P_{_1} \in G \subseteq U_n$. Hence, for every $j = 1, \cdots , n$, we have $|\mu_{_j}|=1$, and so, $\log(\mu_{_j}) = {\bf i} \theta$, with $\theta \in (-\pi, \pi]$. The above equality implies: $\log(\mu_{_j}) = {\bf i}\pi$, so $\mu_{_j} = -1$, for every $j$, and from this: $ P_{_0}^* P_{_1}= -I_{_n} \in G$.

From these arguments, we also easily obtain part (b).
\end{proof}

\begin{prop}\label{diameter}
a) $\delta(\big\langle V \rangle_{_{U_n}}, \phi\big)=\sqrt{n}\, \pi$, \ for every $V \in U_n$ and for every integer $n \geq 1$;

b) $\delta\big(\preccurlyeq Q \succcurlyeq_{_{SU_n}}, \phi\big)=\sqrt{n}\, \pi$, \ \ for every  $Q \in O_n$ and for every even integer $n \geq 2$\ ;

c) $\delta\big(\preccurlyeq Q \succcurlyeq_{_{SU_n}}, \phi\big)=\sqrt{n-1}\, \pi$, \ \ for every  $Q \in O_n$ and for every odd integer $n \geq 1$\ .
\end{prop}

\begin{proof}
Parts a) and b) follow from Corollary \ref{cor-diam} (a) (taking into account also Propositions \ref{UnQ} and \ref{propprinc}), since, in both cases, the groups are connected, SVD-closed and contain $-I_{_n}$.

c) If $n$ is odd, by Remarks \ref{oss-prima-teor-comp} (b), we have $P=Ad_{_A} \big(I_{_1} \oplus (-I_{_{(n-1)}}\big)) \in \ \preccurlyeq Q \succcurlyeq_{_{SU_n}}$ (with $A \in O_n$); hence, from Theorem \ref{distanza} (a), we get
$\delta\big(\preccurlyeq Q \succcurlyeq_{_{SU_n}}, \phi\big)\geq d(I_{_n}, P)=\sqrt{n-1}\, \pi.$
\ Now let $P_{_0}, P_{_1}$ be arbitrary elements of
$\preccurlyeq Q \succcurlyeq_{_{SU_n}}$. Since $n$ is odd, by Proposition \ref{propprinc}, the matrix $P_{_0}^* P_{_1} \in \ \preccurlyeq Q \succcurlyeq_{_{SU_n}}$ has $1$ as eigenvalue; so, from Theorem \ref{distanza} (a), we get $d(P_{_0}, P_{_1}) \leq \sqrt{n-1}\, \pi$\ \ and then (c) holds.
\end{proof}

\begin{rems}\label{diametri}
a) Remembering Remark-Definition \ref{svdgroups} and Lemma \ref{lemma-UH}, from Proposition \ref{diameter}, we deduce the following facts: the diameter of the groups \  $U_n$ \ and \ $U_{(p,n-p)}\cap U_n$

$(p=0, \cdots, n)$ is $\sqrt{n}\, \pi$ \ (for $n \geq 1$); the diameter of  $Sp_n$ and $U_n (\mathbb{H})$ is $\sqrt{2n}\, \pi$ \ (for $n \geq 1$);
the diameter of $SO_n$ and $SO_{(p, n-p)} (\mathbb{C}) \cap U_n \ (p=0, \cdots, n)$ is $\sqrt{n}\, \pi$, for every even integer $n \geq 2$;
while the diameter of the groups $SO_n, \ \ SO_{(p, n-p)} (\mathbb{C}) \cap U_n \ (p=0, \cdots, n) ,$ is equal to $\sqrt{n-1}\, \pi$, when the integer $n \geq 1$ is odd \ (see also \cite[Cor.\,4.12]{DoPe2018a}).

b) \ There are examples of connected closed subgroups $G$ of $U_n$ such that $-I_{_n} \in G$ and 

$\delta(G,\phi) > \sqrt{n} \, \pi$. 
For instance, denoted by $G$ the one-parameter subgroup of $U_2$, given by $\exp(t \Delta)$ ($t \in \mathbb{R}$), where $\Delta$ is the diagonal matrix with eigenvalues $\pi {\bf i}$ and $3 \pi {\bf i}$, it is easy to check that $G$ is compact, not SVD-closed,  $-I_{_2} \in G$ and $\delta(G,\phi) = d(I_{_2}, - I_{_2}) = \sqrt{10} \, \pi$. 
\end{rems}

\section{Generalized principal $\langle V \rangle_{_{\mathfrak{u}_n}}\!\!$--logarithms, with $V \in U_n$}\label{Sect-U-n}

\begin{prop}\label{u-n-plog}
Let $M \in U_n$ and $\zeta \ge 0$ be the multiplicity of $-1$ as eigenvalue of $M$.
Then $\mathfrak{u}_n\!$--$plog(M)$ is disjoint union of $\zeta +1$ compact submanifolds of $\mathfrak{u}_n$, called 
$\mathcal{W}_{_0}, \cdots , \mathcal{W}_{_{\zeta}}$, such that $\mathcal{W}_{_j}$ is diffeomorphic to the complex Grassmannian ${\bf Gr}(j; \mathbb{C}^\zeta)$, for $j = 0 , \cdots , \zeta$.
\end{prop}

\begin{proof}
If $\zeta =0$, the statement is true, since $\mathfrak{u}_n$--$plog(M)$ and ${\bf Gr}(0; \mathbb{C}^0)$ reduce to a point.

Assume now $\zeta \ge 1$. Let us denote the eigenvalues of $M$ as in Remark \ref{val-sing-immag}, with $\theta_{_1}=\pi$ and $\zeta = m_{_1}$. It is well-known that $M$ can be diagonalized by means of a unitary matrix; hence, by Lemma \ref{lemma-plog-coniug}, we can assume  $M =  (\!-I_{_\zeta}) \oplus (\bigoplus\limits_{j=2}^p e^{{\bf i} \theta_{_j}} I_{_{m_{_j}}}\!)$, so that, by Lemma \ref{anticomm} (b), we have $\langle M \rangle_{_{U_n}} =  \ U_{\zeta} \oplus (\bigoplus\limits_{j=2}^p U_{m_{_j}}\!)$.
Let $T$ denote the maximal torus of $U_n$, passing through $M$, consisting of all unitary diagonal matrices, whose Lie algebra is the Cartan subalgebra $\mathfrak{t}$  of $\mathfrak{u}_n$, consisting of all skew-hermitian diagonal matrices (see, for instance, \cite[p.\,98]{Sepa2007}).
Since $|\theta_{_j}| < \pi$, for every  $j \ge 2$, we have that $\mathfrak{t}$--$plog(M)$ is the set of the $2^{\zeta}$ elements of the form $D \oplus (\bigoplus\limits_{j=2}^p {\bf i} \theta_{_{j}} I_{m_{_j}}\!)$, where $D$ is any diagonal matrix of order $\zeta$, having each diagonal element equal to either ${\bf i} \pi$ or $- {\bf i} \pi$. We denote $D_{_j}:=({\bf i}\pi I_{_j}) \oplus (- {\bf i}\pi I_{_{(\zeta-j)}})$ and $L_{_j}:= D_{_j} \oplus (\bigoplus\limits_{j=2}^p {\bf i} \theta_{_{j}} I_{m_{_j}}\!)$, \ so that 
$\langle L_j \rangle_{_{U_n}} = U_j\oplus U_{(\zeta-j)} \oplus(\bigoplus\limits_{j=2}^p U_{m_{_j}}\!)$,\ for $j=0, \cdots , \zeta.$
 
Clearly, each matrix of $\mathfrak{t}$--$plog(M)$ belongs to the $Ad\big(\langle M \rangle_{_{U_n}}\!\big)$-orbit of a unique $L_{_j}$.
Denoted $\mathcal{W}_{_j}:=Ad\big(\langle M \rangle_{_{U_n}}\!\big)(L_{_j})$, 
by Theorem \ref{maximal-torus} we get: 
$\mathfrak{u}_n$--$plog(M) = \bigsqcup\limits_{j=0}^{\zeta} \mathcal{W}_{_j}$, with $\mathcal{W}_{_j}$
compact submanifolds of $\mathfrak{u}_n$, diffeomorphic to 
$\dfrac{\langle M \rangle_{_{U_n}}}{\langle L_{_j} \rangle_{_{U_n}}} = \dfrac{ U_{\zeta} \oplus(\bigoplus\limits_{j=2}^p U_{m_{_j}}\!)}{ U_j\oplus U_{(\zeta - j)} \oplus(\bigoplus\limits_{j=2}^p U_{m_{_j}}\!)} \simeq \dfrac{U_{\zeta}}{U_j \oplus U_{(\zeta -j)}}$, and it is well-known that  
this last homogeneous space is diffeomorphic to the complex Grassmannian ${\bf Gr}(j; \mathbb{C}^\zeta)$, \ for $j = 0 , \cdots , \zeta$.
\end{proof}

\begin{thm}\label{main-thm-par5}
Let $V \in U_n$; denote by $\lambda_{_1}$ (with multiplicity $n_{_1}$), $\cdots, \lambda_{_r}$ (with multiplicity $n_{_r}$) its  distinct eigenvalues, and choose $R \in U_n$ such that $V = Ad_{_R}\big(\bigoplus\limits_{j=1}^r \lambda_{_j} I_{_{n_j}}\big)$. Then

a) $M \in \langle V \rangle_{_{U_n}}$ if and only if $M= Ad_{_R} \big( \bigoplus\limits_{j=1}^r M_{_j}\big)$ , with $M_{_j} \in U_{n_{_j}}$\ , for $j=1, \cdots, r$;

b) if $M= Ad_{_R}\big( \bigoplus\limits_{j=1}^r M_{_j}\big) \in \langle V \rangle_{_{U_n}}$ (with $M_{_j} \in U_{n_{_j}}$), \ 
and $\zeta_{_j} \ge 0$ is the multiplicity of $-1$ as eigenvalue of $M_{_j}$ ($1 \le j \le r$), then the set
$\langle V \rangle_{_{\mathfrak{u}_n}}\!\!\!\!-\!\!\mbox{plog}(M)$ has $\prod\limits_{j=1}^r (\zeta_{_j} +1)$ connected components, called $\mathcal{Z}(k_{_1}, \cdots , k_{_r}\!)$ (for $k_{_j} = 0, 1, \cdots , \zeta_{_j}$ and  $j= 1, \cdots , r$);  each component $\mathcal{Z}(k_{_1}, \cdots , k_{_r}\!)$ is a simply connected compact submanifold of $\mathfrak{u}_n$, diffeomorphic to the product of complex Grassmannians 
$\prod\limits_{j=1}^r {\bf Gr}(k_{_j}; \mathbb{C}^{\zeta_{_j}})$.
\end{thm}

\begin{proof} 
Part (a) follows directly from Proposition \ref{UnQ}. We now prove part (b). By Lemma \ref{lemma-plog-coniug}, we can assume $V= \bigoplus\limits_{j=1}^r \lambda_{_j} I_{n_j}$ (i.e. $R=I_{_n}$) and, so, again by Proposition \ref{UnQ}, we have
$\langle V \rangle_{_{U_n}} = \bigoplus\limits_{j=1}^r U_{n_{_j}} $,  \ \ 
$\langle V \rangle_{_{\mathfrak{u}_n}} = \bigoplus\limits_{j=1}^r \mathfrak{u}_{n_{_j}}$ \ and  \  $M= \bigoplus\limits_{j=1}^r M_{_j}$.
From this,
it follows that $L \in \langle V \rangle_{_{\mathfrak{u}_n}}\!\!$--$plog(M)$ if and only if $L= L_{_1} \oplus \cdots \oplus L_{_r}$, 
where $L_{_j} \in \mathfrak{u}_{n_{_j}}\!\!$--$plog(M_{_j})$, for every $j=1, \cdots, r$. 
This implies that
$\langle V \rangle_{_{\mathfrak{u}_n}}\!\!$--$plog(M) = \bigoplus\limits_{j=1}^r \mathfrak{u}_{n_{_j}}\!\!$--$plog (M_{_j})$.

From Proposition \ref{u-n-plog}, we get that the set $\mathfrak{u}_{n_{_j}}\!\!$--$plog(M_{_j})$ is disjoint union of $\zeta_{_j} +1$ compact submanifolds of $\mathfrak{u}_{n_{_j}}$, called $\mathcal{W}_{_{j0}}, \cdots , \mathcal{W}_{_{j \zeta_j}}$, where $\mathcal{W}_{_{jk}}$ is diffeomorphic to the complex Grassmannian ${\bf Gr}(k; \mathbb{C}^{\zeta_{_j}})$, for every \ $k = 0 , \cdots , \zeta_{_j}$ \ and \ $j=1, \cdots, r$. 
Hence: 

$\langle V \rangle_{_{\mathfrak{u}_n}}\!\!$--$plog(M)= \bigoplus\limits_{j=1}^r \big(\bigsqcup\limits_{k_j =0}^{\zeta_{_j}} \mathcal{W}_{_{jk_j}} \big) = \bigsqcup\limits_{0 \le k_{_1} \le \zeta_{_1}, \cdots , 0 \le k_{_r} \le \zeta_{_r}} \  \bigoplus\limits_{j=1}^r \mathcal{W}_{_{jk_j}},$\  where
each
$\bigoplus\limits_{j=1}^r \mathcal{W}_{_{jk_j}}$ 

\smallskip

is a connected component of $\langle V \rangle_{_{\mathfrak{u}_n}}\!\!$--$plog(M)$ and a compact submanifold of $\mathfrak{u}_n$, diffeomorphic to the product  
$\prod\limits_{j=1}^r {\bf Gr}(k_{_j}; \mathbb{C}^{\zeta_{_j}})$. The total number of these components is $\prod\limits_{j=1}^r (\zeta_{_j} +1)$. Setting $\mathcal{Z}(k_{_1}, \cdots , k_{_r}\!):=\bigoplus\limits_{j=1}^r \mathcal{W}_{_{jk_j}}$ (for all possible indices), we obtain (b).
\end{proof}

\section{Generalized principal  $\preccurlyeq Q \succcurlyeq_{_{\mathfrak{s\!u}_{n}}}\!\!$--logarithms, with $Q \in O_n$}\label{Sect-Q-in-O-n}

\begin{rem}\label{eta0}
By Lemma \ref{lemma-UH}, we have $U_n (\mathbb{H}) = \preccurlyeq \Omega^{\oplus n} \!\!\succcurlyeq_{_{SU_{2n}}}$. Then, arguing as in the proof of Lemma \ref{SUn(V)}, it is easy to show that
any matrix $M \in U_n (\mathbb{H})$ is similar to a real matrix;  so, if $-1$ is an eigenvalue of $M \in U_n (\mathbb{H})$, its multiplicity is even and the eigenvalues of $M$ can be listed as follows: $-1$ with multiplicity $2\mu\geq2$, \ $e^{\pm{\bf i} \eta_{_1}}$ both with multiplicity $\mu_{_1}$, $e^{\pm{\bf i} \eta_{_2}}$ both with multiplicity $\mu_{_2}$, $\cdots$, up to $e^{\pm{\bf i} \eta_{_q}}$ both with multiplicity $\mu_{_q}$ ($q\geq0$), 
where 
$\pi > \eta_{_1} > \eta_{_2} > \cdots > \eta_{_q} \ge 0$, with the agreement that,  
if $\eta_{_q}=0$, the multiplicity of the corresponding eigenvalue \ $1$ \ is \ $2 \mu_{_q}$. In any case we have: \ $\mu+ \sum\limits_{j=1}^q  \mu_{_j} = n$.
\end{rem}

\begin{prop}\label{h-F-n-plog0}
Let $M \in U_n (\mathbb{H})$; denote by $2 \mu\geq0$ the multiplicity of $-1$ as eigenvalue of $M$.
Then $\mathfrak{u}_n (\mathbb{H})$--$plog(M)$ is a simply connected compact submanifold of $\mathfrak{u}_n (\mathbb{H})$, diffeomorphic to the symmetric homogeneous space $\dfrac{U_{\mu} (\mathbb{H})}{U_{\mu}} \simeq \dfrac{Sp_{\mu}}{U_{\mu}}$.
\end{prop}

\begin{proof}
If $\mu=0$ (i.e. if $-1$ is not an eigenvalue of $M$), the statement is true, remembering Notations \ref{notazioni} (a) and Proposition \ref{esist-gplog} (a). Assume now $\mu \geq 1$. It is easy to show that the group 
$T= \big\lbrace \bigoplus\limits_{j=1}^n E_{_{\theta_{_j}}}: \theta_{_1},\cdots \theta_{_n}\in \mathbb{R}\big\rbrace$ is a maximal torus of $U_n (\mathbb{H})$, whose Lie algebra is 
$\mathfrak{t}=\big\lbrace\bigoplus\limits_{j=1}^n 
\theta_{_j}\Omega: \theta_{_1},\cdots \theta_{_n}\in \mathbb{R}\big\rbrace$.
We denote the eigenvalues of $M$ and their multiplicities as in Remark \ref{eta0}; then, by \cite[Thm.\,5.12 (a)]{Sepa2007}, there exists $K \in U_n (\mathbb{H})$ such that  $M=Ad_{_K} \big( (-I_{_{2\mu}}) \oplus ( \bigoplus\limits_{j=1}^q E_{_{\eta_{_j}}}^{\oplus \mu_{_j}} )
\big).$ By Lemma \ref{lemma-plog-coniug}, we can assume $K=I_{_{2n}}$; hence, by Remark \ref{Rodrigues}, \ the set 
$\mathfrak{t}$--$plog(M)$ consists of the $2^\mu$ elements of the form 

$\big(  \bigoplus\limits_{h=1}^\mu (\epsilon_{_h}\pi\Omega)   \big)  \oplus \big( \bigoplus\limits_{j=1}^q 
(\eta_{_j} \Omega)^{\oplus \mu_j} 
\big)$, where each $\epsilon_{_h}$ is either $1$ or $-1$. All these elements belong to the same $Ad\big(\langle M \rangle_{_{U_n\!(\mathbb{H})}}\!\big)$-orbit. Indeed,
 it suffices to remark that the matrix $\Psi({\bf k}) =
\begin{pmatrix}
 0 & -{\bf i} \\ 
 -{\bf i} & 0
 \end{pmatrix}  
 $ satisfies $\Psi({\bf k}) \, \Omega \, \Psi({\bf k})^* = - \Omega$. Hence, by Theorem \ref{maximal-torus}, $\mathfrak{u}_n (\mathbb{H})$--$plog(M)$ is a compact submanifold of $\mathfrak{u}_n (\mathbb{H})$, diffeomorphic to the homogeneous space $\dfrac{\langle M \rangle_{_{U_n\!(\mathbb{H})}}}{\langle L \rangle_{_{U_n\!(\mathbb{H})}}}$, where $L :=(\pi\Omega)^{\oplus \mu}     \oplus \big( \bigoplus\limits_{j=1}^q 
(\eta_{_j} \Omega)^{\oplus \mu_j} 
\big)$.
Recalling Remarks \ref{identificazioni} (c), (d), we get the statement, since we have \ $\langle M \rangle_{_{U_n\!(\mathbb{H})}}\!=  U_{\mu} (\mathbb{H}) \oplus \big( \bigoplus\limits_{j=1}^q \Phi(U_{\mu_{_j}}) \big)$ \ and
\ $\langle L \rangle_{_{U_n\!(\mathbb{H})}}\!= \Phi(U_{\mu}) \oplus \big( \bigoplus\limits_{j=1}^q \Phi(U_{\mu_{_j}}) \big).$
\end{proof}

\begin{rem}\label{casi-visti}
In Remarks \ref{identificazioni} (c), we have seen
 that we have $Ad_{_B}\big(U_n (\mathbb{H})\big)= Sp_n$, with $B \in O_{2n}$; so, by Lemma \ref{lemma-plog-coniug}, we obtain $\mathfrak{sp}_n\!$--$plog(M)= Ad_{B}\big[\mathfrak{u}_n (\mathbb{H})$--$plog(Ad_{_{B^T}}(M))\big]$, for every $M \in Sp_n$.
Hence, by Proposition \ref{h-F-n-plog0}, we conclude that the set $\mathfrak{sp}_n\!$--$plog(M)$ is a simply connected compact submanifold of $\mathfrak{sp}_n$,  diffeomorphic to the symmetric space $\dfrac{Sp_{\mu}}{U_{\mu}}$, where \ $2\mu\geq0$ \ is the multiplicity of \ $-1$ \ as eigenvalue of $M$, \ for every $M \in Sp_n$.
\end{rem}

\begin{prop}\label{so-plog}
Let $M \in \ SO_{(p, n-p)} (\mathbb{C}) \cap U_n$ ($p=0, \cdots, n$) and denote by $2 m \ge 0$ the multiplicity of $-1$ as eigenvalue of $M$.
Then the set $\big(\mathfrak{so}_{(p, n-p)} (\mathbb{C}) \cap \mathfrak{u}_n\big)$--$plog(M)$ is a compact submanifold of $\mathfrak{su}_n$,  diffeomorphic to the homogeneous space $\dfrac{O_{2m}}{U_m}$; hence, if $m \ge 1$, this set has two connected components, both diffeomorphic to the simply connected compact symmetric homogeneous space $\dfrac{SO_{2m}}{U_m}$.
\end{prop}

\begin{proof}
By Lemmas \ref{rem-Jp} and \ref{lemma-plog-coniug}, we can assume $p=n$, so that 
$SO_{(p, n-p)} (\mathbb{C}) \cap U_n = SO_n$, and, in this case, the Proposition has already been proved in 
\cite[\S 3]{DoPe2018a} and in \cite[Thm.\,4.7]{Pe2022}. 
A further proof can be deduced from Theorem \ref{maximal-torus},  but, for the sake of brevity, we omit it.
\end{proof}

\begin{thm}\label{teor-compl}
Let $Q \in O_n$, and assume that $Q$ has, as real Jordan form, the matrix

$\mathcal{J}:=J^{(p, q)} \oplus \big( \bigoplus\limits_{j=1}^h  E_{_{\varphi_{_j}}}^{(\mu_{_j}, \, \nu_{_j})} \big) \oplus \Omega^{\oplus k}$, 
with  $0 < \varphi_{_1} < \varphi_{_2} < \cdots < \varphi_{_h} < \dfrac{\pi}{2}$,

$p+q + 2 \sum\limits_{j=1}^h (\mu_{_j}+\nu_{_j}) + 2k =n$, \  \  $p, q,  k, \mu_{_j}, \nu_{_j} \ge 0$, \ \  $\mu_{_j} + \nu_{_j} \ge 1$,  and choose $A \in O_n$ such that $Q = Ad_{_A}(\mathcal{J}) = A \mathcal{J}A^T $. \ Let $Z$ be the $n \times n$ unitary matrix defined by 

$Z :=\ A \bigg(W_{(p, q)}\oplus \bigg[\bigoplus\limits_{j=1}^h W_{(2\mu_{_j}, 2\nu_{_j})}\bigg] \oplus I_{_{2k}}\bigg)$.
Then 

a) $M \in \ \preccurlyeq  Q  \succcurlyeq_{_{SU_n}}\ $ if and only if \ $M = Ad_{_Z} \bigg[N\oplus \bigg(\bigoplus\limits_{j=1}^h M_{_j} \bigg) \oplus R\bigg]$, where 

$N \in SO_{(p+q)}$, \ $R \in U_k (\mathbb{H})$ \ and \ $M_{_j} \in U_{(\mu_{_j}+ \nu_{_j})}$, \ for $j=1, \cdots, h$.

\smallskip

b) If $M = Ad_{_Z} \bigg[N\oplus \bigg(\bigoplus\limits_{j=1}^h M_{_j} \bigg) \oplus R\bigg] \in \ \preccurlyeq  Q  \succcurlyeq_{_{SU_n}}$, denote by $2m \geq 0$ the multiplicity of $-1$ as eigenvalue of $N$, by $\zeta_{_j} \geq 0$ the multiplicity of $-1$ as eigenvalue of $M_{_j}$ (for $1 \le j \le h$) and by $2\mu \geq0$ the multiplicity of $-1$ as eigenvalue of $R$.
Then we have
\begin{center}
$\preccurlyeq Q \succcurlyeq_{_{\mathfrak{s\!u}_n}}\!\!$--$plog(M) = \bigsqcup\limits_{0 \le l_{_1} \le \zeta_{_1}, \cdots , 0 \le l_{_h} \le \zeta_{_h}}  
\mathcal{V}(l_{_1}, \cdots , l_{_h})$,
\end{center} 
where each $\mathcal{V}(l_{_1}, \cdots , l_{_h})$ is a compact submanifold of $\mathfrak{su}_n$, diffeomorphic to the product $\dfrac{O_{2m}}{U_m} \times \bigg[ \prod\limits_{j=1}^h {\bf Gr}\big(l_{_j}; \mathbb{C}^{\zeta_{_j}}\big) \bigg] \times \dfrac{Sp_{\mu}}{U_{\mu}}$. 

If $-1$ is not an eigenvalue of $N$ (i.e. if $m = 0$), then each $\mathcal{V}(l_{_1}, \cdots , l_{_h})$ is connected and $\preccurlyeq Q \succcurlyeq_{_{\mathfrak{s\!u}_n}}\!\!$--$plog(M)$ has \ $\prod\limits_{j=1}^h (\zeta_{_j} +1)$ components; while, if $-1$ is an eigenvalue of $N$ (i.e. if $m \ge 1$), then each $\mathcal{V}(l_{_1}, \cdots , l_{_h})$ has two connected components, both diffeomorphic to 

\smallskip

$\dfrac{SO_{2m}}{U_m} \times \bigg[ \prod\limits_{j=1}^h {\bf Gr}\big(l_{_j}; \mathbb{C}^{\zeta_{_j}}\big) \bigg] \times \dfrac{Sp_{\mu}}{U_{\mu}}$, so $\!\preccurlyeq Q \succcurlyeq_{_{\mathfrak{s\!u}_n}}\!\!$--$plog(M)$ has\ $2 \prod\limits_{j=1}^h (\zeta_{_j} +1)$ components.
In any case, all components of $\preccurlyeq Q \succcurlyeq_{_{\mathfrak{s\!u}_n}}\!\!$--$plog(M)$ are simply connected, compact and diffeomorphic to a symmetric homogeneous space.
\end{thm}

\begin{proof}
Part (a) follows directly from Proposition \ref{propprinc}.
\ \ By Lemma \ref{lemma-plog-coniug}, we can assume 

$\preccurlyeq Q \succcurlyeq_{_{SU_n}}\ = SO_{(p+q)}\oplus \bigg[\bigoplus\limits_{j=1}^h U_{(\mu_{_j}+ \nu_{_j})} \bigg]\oplus U_k (\mathbb{H})$ \ \ and \ $M=N\oplus \bigg(\bigoplus\limits_{j=1}^h M_{_j} \bigg) \oplus R.$

\smallskip

Therefore, arguing as in the proof of Theorem \ref{main-thm-par5}, \ we get
\ \ $\preccurlyeq Q \succcurlyeq_{_{\mathfrak{s\!u}_n}}\!\!$--$plog(M) = $

\smallskip

$\bigg[\mathfrak{so}_{(p+q)}$--$plog(N)\bigg] \oplus \bigg[ \bigoplus\limits_{j=1}^h \mathfrak{u}_{(\mu_{_j}+ \nu_{_j})}$--$plog(M_{_j})
\bigg] \oplus \bigg[\mathfrak{u}_k (\mathbb{H})$--$plog(R) \bigg]$. 

\smallskip

Hence we get (b), by means of Propositions \ref{h-F-n-plog0}, \  \ref{so-plog} and \ref{u-n-plog}, via Remarks \ref{confronta-Higham} (b).
\end{proof}

\end{document}